\documentclass[reqno,11pt]{amsart}
\usepackage[left=3cm,right=3cm,top=3cm,bottom=2.5cm, marginparwidth=2cm]{geometry}
\usepackage{graphicx} 
\usepackage[utf8]{inputenc}
\usepackage[T1]{fontenc}
\usepackage{amsmath,amssymb}
\usepackage{amsthm}
\usepackage[all]{xy}
\usepackage{amsfonts}
\usepackage{hyperref}

\newcounter{ct}
\counterwithin{ct}{section}
\newtheorem{lem}[ct]{Lemma}
\newtheorem{prop}[ct]{Proposition}
\newtheorem{thm}[ct]{Theorem}
\newtheorem{cor}[ct]{Corollary}
\newtheorem{ques}[ct]{Question}
\theoremstyle{definition}
\newtheorem{df}[ct]{Definition}
\theoremstyle{remark}
\newtheorem{rmk}[ct]{Remark}
\newtheorem{ex}[ct]{Example}
\usepackage{nicematrix}
\newcommand{\C}{\mathbb{C}}
\newcommand{\R}{\mathbb{R}}
\newcommand{\Q}{\mathbb{Q}}
\newcommand{\Z}{\mathbb{Z}}
\newcommand{\SO}{\mathrm{SO}}
\newcommand{\SL}{\mathrm{SL}}

\newcommand{\HSC}{\mathrm{HSC}}
\newcommand{\Hom}{\mathrm{Hom}}

\newcommand{\Id}{\mathrm{Id}}
\newcommand{\dd}{\mathrm{d}}
\newcommand{\ddbar}{\partial\overline{\partial}}

\newcommand{\HK}{hyperkähler }
\newcommand{\Ram}{\mathrm{Ram}}
\newcommand{\Kah}{\mathrm{K\ddot{a}h}}

\newcommand{\pr}{\mathrm{pr}}

\title{Hyperbolicity properties of moduli spaces of marked hyperkähler manifolds}
\author[B. Philippe]{Bastien Philippe}
\address{Université de Lorraine, CNRS, IECL, F-54000 Nancy, France}
\email{bastien.philippe@univ-lorraine.fr}

\begin{document}

\begin{abstract}
    We study the hyperbolicity properties of moduli spaces of marked \HK manifolds along directions corresponding to families having positivity properties for their Hodge bundle. In particular, we show that the Kobayashi pseudo-distance computed using disks tangent to these directions vanishes. As an intermediate step, we establish the existence of families of marked hyperkähler manifolds over arbitrary curves having a prescribed period map to the corresponding period domain. This generalizes a recent theorem of Greb and Schwald \cite{GrebSchwald} to the case of \HK manifolds of arbitrary dimensions and non-necessarily compact curves. Finally, using Nevanlinna theory, we establish restrictions on families of \HK manifolds over $\C$ having positive Hodge bundle.
\end{abstract}

\maketitle

\section{Introduction}\label{Section intro}

Let $B$ be a complex manifold and let $\pi : \mathcal{X} \rightarrow B$ be a smooth family of compact complex manifolds over $B$. In the last few years, the question of studying the complex hyperbolicity properties of $B$ induced by the existence of the family has attracted a lot of attention. 

On one side, one can study the hyperbolicity properties coming from the induced variation of Hodge structures (VHS for short) on $B$. It was already noticed by Griffiths and Schmid in \cite[Corollary 9.4]{GriffithsSchmid69} that any \textit{polarized} VHS on $\C$ must be trivial. As a consequence, any entire curve traced inside $B$ must lie on the locus $Z \subset B$ where the VHS degenerates, i.e. the locus where the period map fails to be immersive. In particular, if $Z$ is a proper subvariety, then $B$ is \textit{Brody hyperbolic} modulo $Z$. In the case where $B$ is quasi-projective, even stronger hyperbolicity properties are known (see for example \cite{BrotbekBrunebarbe,DengBigPicardVHS}). 

In another direction, under the assumption that $B$ is quasi-projective and that $\pi$ is a family of \textit{polarized} manifolds of a certain type, other results have been obtained on the hyperbolicity of $B$. For example, if the fibers of $\pi$ are canonically polarized, Calabi-Yau or manifolds with semi-ample canonical bundle, then under the assumption that $\pi$ has maximal variation, i.e. the associated Kodaira-Spencer map is injective at every point of $B$, $B$ is known to be Kobayashi hyperbolic (\cite{Berndtsson_Păun_Wang_2022,ToYeung,DengHypPolarized}). Similarly, if $\pi$ parametrizes manifolds of the same type and the induced moduli map (as constructed in \cite{viehweg2012quasi}) is quasi-finite over its image, $B$ is Brody hyperbolic (\cite{ViehwegZuoHyp,DengHypPolarized,Deng2019PicardTF}).

The goal of this work is to investigate the case where $\pi$ \textit{is not necessarily polarized}. We will restrict our study to the case of \HK manifolds. In this case, the local Torelli Theorem allows one to use Hodge variational techniques mentioned above to deduce that, if $\pi$ is polarized (or even Kähler) and has maximal variation, then $B$ is Brody and Kobayashi hyperbolic. Without any assumption on $\pi$, it can happen that $B$ fails to be hyperbolic. Indeed, every \HK manifold is a member of a \textit{twistor family}, which is a family with everywhere maximal variation and whose base is $\mathbb{P}^1$, see Example \ref{exemple twistor}. 

A crucial tool in establishing the hyperbolicity results mentioned above is the use of positivity properties of certain direct image sheaves. Given a smooth family $\pi : \mathcal{X} \rightarrow B$, one can consider the sheaf $\pi_*\omega_{\mathcal{X}/B}$. This sheaf is known to be the sheaf of sections of a holomorphic vector bundle whose fiber over a point $b \in B$ is naturally isomorphic to $H^0(X_b, K_{X_b})$. This bundle is called the \textit{Hodge bundle} of the family. On this bundle, one can consider the $L^2$ metric given by taking the $L^2$ norm on each space $H^0(X_b, K_{X_b})$. Under the assumption that $\pi$ is Kähler, the Hodge bundle endowed with this metric is known to Nakano semi-positive (\cite{BerndtssonPositive}). 

In the case of families of manifolds with trivial canonical bundle, such as \HK manifolds, the Hodge bundle is a line bundle. For twistor families, one can show that the Hodge bundle is isomorphic to $\mathcal{O}_{\mathbb{P}^1}(-2)$ and that the $L^2$ metric is isomorphic to one induced by a positive hermitian form on $\C^2$ (Example \ref{exemple twistor}). One is thus left to wonder if this lack of positivity of the Hodge bundle is the only source for the non-hyperbolicity of the base. More precisely, one can consider the following question. 

\begin{ques}\label{question intro}
    Let $\pi : \mathcal{X} \rightarrow B$ be a family of \HK manifolds. Assume that $\pi$ has maximal variation. Assume moreover that the Hodge bundle, endowed with its natural $L^2$ metric, is positive on $B$. Is $B$ Brody or Kobayashi hyperbolic ? 
\end{ques}

Let $\pi : \mathcal{X} \rightarrow B$ be as in Question \ref{question intro}. Let $f : \Delta \rightarrow B$ be a non-constant holomorphic map, where $\Delta$ is a disk. One can construct the fibered product $\pi_\Delta : \mathcal{X}_\Delta \rightarrow \Delta$ which is a smooth family of \HK manifolds over $\Delta$ with generically positive Hodge bundle. Since $\Delta$ is simply connected and $\pi$ has maximal variation, this family will induce a non-constant moduli map $m : \Delta \rightarrow \mathcal{M}_\Lambda$ where $\mathcal{M}_\Lambda$ is the moduli space for $\Lambda$-marked pairs of \HK manifolds for some lattice $\Lambda$. We refer to sections \ref{section familles} and \ref{section espace de modules} for the definitions of $\Lambda$-marking and of $\mathcal{M}_\Lambda$. 

The above shows that the Kobayashi pseudo-distance on $B$ is closely related to the modified Kobayashi pseudo-distance on $\mathcal{M}_\Lambda$ computed using only disks arising as moduli maps of families with (generically) positive Hodge bundle. More precisely, we will consider disks that carry families whose period map is tangent to directions of the period domains along which a natural metric is positive (see Section \ref{section positive distance} for more details). We refer to this modified Kobayashi pseudo-distance on $\mathcal{M}_\Lambda$ as the \textit{positive Kobayashi pseudo-distance}. We now state the main result of our work. 

\begin{thm}[$\subset$ Theorem \ref{thm distance koba espace de mod}]\label{thm dist koba intro}
    Assume that $\mathrm{rk}(\Lambda) \geqslant 4$. Then the positive Kobayashi pseudo-distance vanishes on every connected component of $\mathcal{M}_\Lambda$. 
\end{thm}

It is an easy consequence of our approach that for families of \HK manifolds with $b_2 = 3$, the curvature of the Hodge bundle is always non-positive, see Example \ref{exemple twistor}. In particular, in this case the positive pseudo-Kobayashi pseudo distance that we consider would be equal to $+\infty$. 

In order to prove Theorem \ref{thm dist koba intro}, we adopt the following strategy. We start by showing the analogous statement for the period domain $D_\Lambda$ corresponding to $\mathcal{M}_\Lambda$. We show in Section \ref{section metric GS} that $D_\Lambda$ carries a non-degenerate pseudo-Kähler metric $\omega_D$ of Lorentzian type such that the curvature of the Hodge bundle on the base of a family of $\Lambda$-marked \HK manifold is the pullback of $\omega_D$ by the period map. By construction, $\omega_D$ restricts to (a multiple of) the Bergman metric on subdomains of $D_\Lambda$ parameterizing polarized Hodge structures. We then study period domains of $\mathrm{K3}$-type (see Definition \ref{def period domain}) that are of dimension 2 to describe explicitly the metric $\omega_D$ for these domains. Using this description, we show the vanishing of the positive Kobayashi pseudo-distance on these two-dimensional domains. We then deduce the vanishing on $D_\Lambda$ by connecting any pair points by a chain of such domains.

In order to deduce the vanishing on $\mathcal{M}_\Lambda$ from the vanishing on $D_\Lambda$, we establish the following, which we believe is a result of independent interest. 

\begin{thm}[= Theorem \ref{thm lift}]\label{thm lift intro}
    Let $\mathcal{M}_\Lambda^0$ be a connected component of $\mathcal{M}_\Lambda$. Let $C$ be a (non-necessarily compact) connected smooth curve and let $x_0 \in C$. Let $f : C \rightarrow D_\Lambda$ be a holomorphic map. Let $(X_0,\varphi_0) \in \mathcal{M}_\Lambda^0$ be a marked \HK manifold such that $\mathcal{P}(X_0, \varphi_0) = f(x_0)$. There exists a family $\mathcal{X} \rightarrow C$ and a marking $\varphi : \mathrm{R}^2\pi_*\underline{\Z} \rightarrow \underline{\Lambda}$ such that \begin{itemize}
        \item $f$ is the period map induced by this marked family ;
        \item $(X_0, \varphi_0) \simeq (X_{x_0}, \varphi_{x_0})$.
    \end{itemize}
\end{thm}

This provides a generalization of a recent theorem of Greb and Schwald \cite[Theorem 4.4]{GrebSchwald}, where the above is shown in the case where $\mathcal{M}_\Lambda$ is the moduli space of marked $\mathrm{K3}$ surfaces and $C$ is compact. In order to prove Theorem \ref{thm lift intro}, we follow the strategy of Greb and Schwald, adapting arguments to the \HK case as follows. 

We start by noticing that by \cite{MarkmanUniversalFamilies}, there exists a family on $\mathcal{M}_\Lambda$ that satisfies a modified universal property. This property implies that the period map of this family coincides with $\mathcal{P}$. Using the description of the fibers of the period map given by Markman in \cite{MarkmanSurvey}, Theorem \ref{thm lift intro} amounts to showing that one can choose for every $x \in C$ a Kähler-type chamber over $f(x)$ in a continuous way. In the case of $\mathrm{K3}$ surfaces, Greb and Schwald use the description of these chambers in terms of $(-2)$ classes. In order to treat the higher-dimensional \HK case, we use the analogous description in terms of MBM classes introduced by Amerik and Verbitsky in \cite{AmerikVerbitsky}. As a consequence, we need to find another way to ensure that over every point inside $f(C)$, there exists a Kähler-type chamber intersecting a set of MBM classes with prescribed signs. The approach of Greb and Schwald ultimately relies on the action of a Weyl group on the set of chambers. This strategy no longer applies when using general MBM classes. As an alternative, we borrow ideas from the theory of Mumford-Tate domains and show that $f(C)$ is contained in a subspace of $D_\Lambda$ which is homogeneous under the action of a group preserving the MBM classes in question (Lemma \ref{lem domaine de Mumford Tate}). Finally, we adapt Greb and Schwald's original topological argument in order to include the case of non-compact curves. 

\medskip

We now turn to the question of the Brody hyperbolicity of $B$. We start by noticing that as a direct consequence of computation of Section \ref{section domaine dim 2} and Theorem \ref{thm lift intro}, for any \HK manifold $X$, there exists a family $\pi : \mathcal{X} \rightarrow \C$ with flat Hodge bundle and such that $\pi^{-1}(0) \simeq X$ (Corollary \ref{coro familles plates}). 

Let $D$ be a period domain of $\mathrm{K3}$-type and let $\mathbb{P}(\Omega_D^1)^+$ be the open subset of $\mathbb{P}(\Omega_D^1)$ formed by lines inside $T_D$ on which $\omega_D$ restricts to a positive metric. One the one hand, the local Torelli theorem implies that the Brody hyperbolicity of $B$ would follow from the non-existence of entire curves tangent to lines inside $\mathbb{P}(\Omega_D^1)^+$. On the other, if any such entire curve were to exist, we could use Theorem \ref{thm lift intro} to construct a family over $\C$ with (generically) maximal variation and positive Hodge bundle. In particular, the Brody part of Question \ref{question intro} boils down to the following. 

\begin{ques}\label{question 2 intro}
    Let $D$ and $\mathbb{P}(\Omega_D^1)^+$ be as above. Does there exist a non-constant holomorphic map $f : \C \rightarrow D$ such that $f'(\C) \subset \mathbb{P}(\Omega_D^1)^+$ ?
\end{ques}

By comparing Nevanlinna's characteristic functions associated to different metrics on $D$, we obtain the following. 

\begin{thm}[Corollary \ref{coro Nevanlinna 0jet} + Proposition \ref{prop nevanlinna 1jet}]\label{thm Nevanlinna 1 intro}
    For all $\eta >0$, there exists open subsets $D_\eta \subset D$ and $D_{1,\eta} \subset \mathbb{P}(\Omega_D^1)^+$ such that \begin{itemize}
        \item for all $\alpha > \beta$, $D_\alpha \subset D_\beta$ and $D_{1,\alpha} \subset D_{1,\beta}$ ;
        \item $\bigcup_{\eta >0} D_\eta = D$ and $\bigcup_{\eta >0} D_{1,\eta} = \mathbb{P}(\Omega_D^1)^+$,
    \end{itemize}
    satisfying the following condition. For all non-constant holomorphic map $f : \C \rightarrow D$ such that $f^*\omega_D \geqslant 0$ and for all $\varepsilon >0$, there exists $x_0, x_1 \in \C$ such that $f(x_0) \notin D_\varepsilon$ and $f'(x_1) \notin D_{1,\varepsilon}$.
\end{thm}

In Theorem \ref{thm Nevanlinna 1 intro}, the restrictions on potential entire curves can be seen as depending on their lifting to Demailly-Semple's jet bundles $P_0D = D$ and $P_1D = \mathbb{P}(\Omega_D^1)$ (we refer to \cite{DemHyp} for the definition of these bundles). We will now express restrictions on the lifting of such an entire curves to the second stage of Demailly-Semple's jet bundle. 

Classically, hyperbolicity results using metric techniques use the fact that the holomorphic sectional curvature of a positive definite metric decreases on submanifolds. In our situation, this is no longer the case since the metric $\omega_D$ is not positive definite. Using methods from \cite{BrotbekBrunebarbe}, we establish an inequality reminiscent of the Second Main Theorem in Nevanlinna theory, with an added error term encompassing the potential lack of a curvature decrease. We refer to Section \ref{sec Nevanlinna} for the notations used below. 

\begin{thm}[= Theorem \ref{thm SMT}]
    Let $f : \C \rightarrow D$ be an entire curve such that $f'(\C) \subset \mathbb{P}(\Omega_D^1)^+$. There exists $\gamma >0$ and a $(1,1)$-form $\sigma_f$, whose value at any point $x \in \C$ depends only on the point $f^{(2)}(x) \in P_2D$, such that 
    \[ \gamma T_{f, \omega_D}(r) + T_{\sigma_f}(r) \leqslant_{\mathrm{exc}} O(\log(r) + \log(T_{f, \omega_D}(r))),\]
    where the subscript $\mathrm{exc}$ stands for the fact that this inequality is true outside a subset of finite Lebesgue measure in $[1, + \infty [$.
\end{thm}

{\bf Acknowledgments.} This work is part of the author's PhD thesis. He wishes to thank his advisors Benoît Cadorel and Matei Toma for long discussions and fruitful suggestions during the preparation of this work. 

\section{Preliminaries}

\subsection{Families of hyperkähler manifolds}\label{section familles}

We start by recalling basic facts about \HK manifolds. We refer the reader to \cite{GrossJoyceHuybrechts} for more details. A \textit{\HK manifold} $X$ is a simply connected compact Kähler manifold such that $H^{2,0}(X) = H^0(X, \Omega^2_X) = \C \cdot \sigma$, where $\sigma$ is an everywhere non-degenerate symplectic holomorphic two-form. The second cohomology group $H^2(X, \Z)$ of a \HK manifold $X$ is endowed with a non-degenerate bilinear form $q_X$ called the \textit{Beauville-Bogomolov} form. This form satisfies the relation 
\[ q_X(\alpha, \alpha)^n = c\int_X \alpha^{2n},\]
for all $\alpha \in H^2(X, \Z)$, where $2n = \dim(X)$ and $c \in \R$ is a positive constant, called the \textit{Fujiki constant}, depending only on the deformation type of $X$. Consider the Hodge decomposition $H^2(X, \C) = H^{2,0}(X) \oplus H^{1,1}(X) \oplus H^{0,2}(X)$ on the second cohomology group of $X$. Then this decomposition is orthogonal for $q(\bullet, \overline{\bullet})$ and one has the following properties : \begin{itemize}
    \item $q(\bullet, \overline{\bullet})$ is of signature $(3,b_2(x)-3)$ ;
    \item $q(\sigma, \overline{\sigma})> 0, \quad \forall \sigma \in H^{2,0}(X) \oplus H^{0,2}(X)$ ;
    \item $q(\alpha, \overline{\alpha}) > 0$ for any Kähler class $\alpha$.
\end{itemize} 

Let $\pi : \mathcal{X} \rightarrow B$ be a family of \HK manifolds, i.e. $\pi$ is a smooth and proper map $\mathcal{X} \rightarrow B$ whose fibers are \HK manifolds. Consider the sheaf $\mathrm{R}^2\pi_*\underline{\Z}$, which is a local system on $B$. Because the Beauville-Bogomolov form depends only on the deformation type of the manifold, it defines a bilinear form $q : \mathrm{R}^2\pi_*\underline{\Z} \times \mathrm{R}^2\pi_*\underline{\Z} \rightarrow \underline{\Z}$. Let $(\mathcal{H}_\R, \nabla, q)$ be the flat bundle associated to $(R^2\pi_*\underline{\Z}) \otimes \underline{\R}$ endowed with the flat metric associated to $q$, that we still denote by $q$. The fiber over a point $t \in B$ of $\mathcal{H}_\R$ identifies with $H^2(X_t, \R)$ endowed with the Beauville-Bogomolov form. It is known (see for example \cite{Carlson_Müller-Stach_Peters_2017}) that the fiberwise Hodge filtration on $\mathcal{H}_\C := \mathcal{H}_\R \otimes \C$ forms a filtration by holomorphic subbundles of $\mathcal{H}_\C$ that we denote by $\mathcal{F}^\bullet$. The piece $\mathcal{F}^2$, i.e. the sub-bundle of $\mathcal{H}_\C := \mathcal{H}_\R \otimes \C$ whose fiber over a point $t \in B$ is $H^0(X_t, \Omega_{X_t}^2)$, identifies naturally with $\pi_*\Omega^2_{\mathcal{X}/B}$.

\begin{df}
    The bundle $\mathcal{F}^2 = \pi_*\Omega^2_{\mathcal{X}/B}$ is called the \textit{Hodge bundle} of $\pi$. 
\end{df}

Because the fibers of $\pi$ are hyperkähler, the Hodge bundle is a line bundle. Moreover, it is naturally endowed with a hermitian metric induced by the Beauville-Bogomolov metric. This hermitian metric on $\mathcal{F}^2$ will be noted by $h_{BB}$ and will also be called the Beauville-Bogomolov metric.

Consider the natural morphism $\wedge^n : (\Omega^2_{\mathcal{X}/B})^{\otimes n} \rightarrow K_{\mathcal{X}/B}$. The induced morphism on the pushforwards $\pi_*\wedge^n : \pi_*(\Omega^2_{\mathcal{X}/B})^{\otimes n} \rightarrow \pi_*K_{\mathcal{X}/B}$ is an isomorphism. Indeed, on the fiber over $b$ this morphism identifies with $(f_*\wedge^n)_b : H^0(X_b, (\Omega^2_{X_b})^{\otimes n}) \rightarrow H^0(X_b, K_{X_b})$ and for all $b \in B$, $H^0(X_b, \Omega^2_{X_b})$ is generated by a symplectic form. Moreover, using the definition of the Beauville-Bogomolov form given in \cite[Definition 22.8]{GrossJoyceHuybrechts} and the fact that the Fujiki constant is invariant under deformation, one has that the natural $L^2$ metric defined on $\pi_*K_{\mathcal{X}/B}$ is proportional to the $n$-th power of the Beauville-Bogomolov metric under the above identification. We will say that the Hodge bundle is \textit{positive} (resp. \textit{semi-positive}), if it is positive (resp. semi-positive) as a hermitian line bundle when endowed with the Beauville-Bogomolov metric. Clearly, the positivity of the Hodge bundle as we have defined it is equivalent to the positivity of $\pi_*K_{\mathcal{X}/B}$ endowed with the $L^2$ metric.

Let $(\Lambda,q)$ be a lattice, i.e. a torsion-free $\Z$-module of finite type endowed with an integral bilinear form, and assume that $q$ has signature $(3,p)$. A \textit{$\Lambda$-marked pair} is a pair $(X, \varphi)$ where $X$ is a \HK manifold and $\varphi$ is an isometry from $H^2(X, \Z)$ endowed with the Beauville-Bogomolov form to $(\Lambda,q)$. A marking on $\pi$ is an isomorphism $\varphi : \mathrm{R}^2\pi_*\underline{\Z} \rightarrow \underline{\Lambda}$ such that, for all $b \in B$, $\varphi_b : H^2(X_b, \Z) \rightarrow \Lambda$ is a marking on $X_b$. Such a marking exists, for example, if $B$ is simply connected. 

Assume that there exists a marking $\varphi$ on $\pi$. Such a marking gives rise to a map 
\[ \mathcal{P} : B \rightarrow \mathbb{P}(\Lambda_\C), \quad b \mapsto \varphi(\mathcal{F}^2_b).\]
Since $\mathcal{F}^2$ is a holomorphic subbundle of $\mathcal{H}_\C$ this map is holomorphic. Moreover, by the properties of the Beauville-Bogomolov form, this morphism takes its values inside the domain
\[ D_\Lambda := \{\sigma \in \mathbb{P}(\Lambda_\C) \mid q(\sigma, \sigma) = 0, q(\sigma, \overline{\sigma}) > 0 \},\]
called the \textit{the period domain} associated to $\Lambda$. Such period domains will be studied further in Section \ref{Section period domain}.

\subsection{Moduli space of marked \HK manifolds}\label{section espace de modules}

Let $(\Lambda,q)$ be a lattice as in the previous section. Consider the following set
\[\mathcal{M}_\Lambda := \{(X,\varphi) \text{ $\Lambda$-marked pair} \}/\simeq\]
where $(X, \varphi) \simeq (X', \varphi')$ if and only if there exists an isomorphism $\psi : X \rightarrow X'$ such that $\varphi' = \varphi \circ \psi^*$. The set $\mathcal{M}_\Lambda$ is called the \textit{moduli space of $\Lambda$-marked \HK manifolds}. Let us recall some properties of this moduli space. We refer to \cite{HuybrechtsBourbaki} and the references contained within for more details. In what follows, we always assume that $\Lambda$ is such that $\mathcal{M}_\Lambda$ is non-empty. 

By construction, one has a map $\mathcal{P} : \mathcal{M}_\Lambda \rightarrow D_\Lambda$ that associates to a pair $(X, \varphi)$ its period $\varphi_\C(H^{2,0}(X))$. Let $X$ be a \HK manifold. Since $X$ is hyperkähler, its Kuranishi family $\pi : \mathcal{X} \rightarrow \mathrm{Def}(X)$ is unobstructed and universal for all of its fibers. This means the following : \begin{itemize}
    \item $\mathrm{Def}(X)$ is a germ of complex manifold ;
    \item for any $t \in \mathrm{Def}(X)$, if we denote by $X_t$ the fiber of $\pi$ over $t$, any germ of deformation $\pi' : \mathcal{X}' \rightarrow B$ of $X_t$ is obtained by pullback along a unique holomorphic map $B \rightarrow \mathrm{Def}(X)$ of $\pi$.
\end{itemize}
Since $\mathrm{Def}(X)$ is a germ of complex manifold, we can assume that it is simply connected, and thus, any marking of $X$ to $\Lambda$ extends to a marking of $\pi$. One can show that the associated period map $\mathcal{P} : \mathrm{Def}(X) \rightarrow D_\Lambda$ is an isomorphism onto an open subset of $D_\Lambda$. This fact is usually referred to as the local Torelli Theorem.

\begin{prop}[\cite{HuybrechtsBourbaki} Proposition 4.3]\label{prop construction espace de modules}
    The set $\mathcal{M}_\Lambda$ has the structure of a complex manifold characterized by the following : for any point $p := (X, \varphi)$, $\mathrm{Def}(X)$ embeds holomorphically as a neighborhood of $p$ in $\mathcal{M}_\Lambda$. 
\end{prop}

Proposition \ref{prop construction espace de modules} implies that the map $\mathcal{P}$ is holomorphic and \textit{étale}, i.e. is a local isomorphism. Although $\mathcal{M}_\Lambda$ is a manifold, it is in general not Hausdorff. Moreover, for \HK manifolds of dimension strictly greater than 2, there might not exist a universal family on $\mathcal{M}_\Lambda$ in the usual sense, as it does in the case of $\mathrm{K3}$ surfaces. This is due to the existence of non-trivial automorphisms acting trivially on the second cohomology group. However, Markman has shown in \cite{MarkmanUniversalFamilies} that there exists a marked family over $\mathcal{M}_\Lambda$ that satisfies a modified universal property. This property implies that the period map of this family coincides with $\mathcal{P}$.  

\subsection{Pseudo-hermitian metrics}\label{section semi-herm met}

In this section, we gather a few basic facts on pseudo-hermitian metrics on a vector bundle. In what follows, $X$ will be an arbitrary complex manifold.

\begin{df}
    Let $E$ be a vector bundle on $X$. A \textit{pseudo-hermitian metric} on $E$ is a hermitian metric on $E$ that is non-degenerate on every fiber of $E$. 
\end{df}

Notice that in the previous definition we do not require $h$ to be positive definite, hence the name pseudo-hermitian. Let $E$ be a holomorphic vector bundle on $X$ endowed with a pseudo-hermitian metric $h$. As already noticed by Griffiths in \cite{Gri70}, even if $h$ is not positive definite there still exists a unique connection $\nabla$ on $E$ such that $h$ is flat for $\nabla$ and $\nabla^{0,1} = \overline{\partial}$. As in the positive definite case, we will call this connection the Chern connection. Let $\Theta(E,h) := \nabla^2$ be the curvature of $\nabla$.

Let $F$ be a subbundle of $E$ on which $h$ restricts to a non-degenerate metric. The second fundamental form $\beta$ of the connection $D$ associated to $F$ is defined by 
\[\beta(V) \cdot s := \pr_{F^{\perp}}(\nabla_{V}(s))\]
for $V$ a vector field on $X$ and $s$ a smooth section of $F$ (notice that since the restriction of $h$ to $F$ is non degenerate, $F$ and $F^{\perp}$ are in direct sum). One has $\beta \in \mathcal{A}^{1,0}(\Hom(F,F^{\perp}))$. Since $h$ is non-degenerate on $F$, one can consider $\Theta(F,h)$. Arguing in the same way as for positive definite hermitian metrics (see for example \cite[Section V.14]{DemaillyCompAnDiffGeo}), one obtains the following. 

\begin{prop}\label{courbure induite}
    The curvature $\Theta(F,h)$ satisfies 
    \[\Theta(F,h) = \Theta(E,h)|_F - \beta \wedge \beta^*.\]
\end{prop}

Let us now assume that $E = T_X$. We then call $h$ a pseudo-hermitian metric on $X$. Such metrics are in correspondence with non-degenerate real $(1,1)$-forms on $X$. The metric $h$ will be called \textit{pseudo-Kähler} if the associated real $(1,1)$-form $\omega$ is closed. We will often refer to $\omega$ as the metric itself.

\begin{df}
    Let $x \in X$ and $u \in T^{1,0}_{X,x}$ be a non-isotropic vector for $h$. The \textit{holomorphic sectional curvature} of $u$ for the metric $h$ is
    \[\HSC_h(u) := \frac{h(\Theta(T_X,h)(u, \overline{u})u, u)}{h(u,u)^2}.\]
\end{df}

We notice that, in the previous definition, $\HSC(u)$ depends only on the line generated by $u$. We therefore extend the above definition in the case where $u$ denotes only a line in $T^{1,0}_{X,x}$, and will use the same notation.

\section{Period domains of $\mathrm{K3}$-type}\label{Section period domain}

In this section, we study some aspects of the period domains of \HK manifolds introduced in Section \ref{section familles}. The properties that we will study in this section are mostly of differential-geometric nature and don't depend on the integral aspect of the variation of Hodge structures considered. Therefore, we present our study for $\R$-variation of Hodge structures, even though our applications will concern $\Z$-variation of Hodge structures. 

\subsection{Variation of Hodge structures of $\mathrm{K3}$-type}

We start by recalling the notion of variation of Hodge structures. We refer to \cite{Carlson_Müller-Stach_Peters_2017} for more details. In what follows, we fix a complex manifold $B$.

\begin{df}
    A \textit{$\mathbb{R}$-variation of Hodge structure} (or \textit{$\R$-VHS} for short) of weight $n$ on $B$ consists of the data $(V, \mathcal{F}^\bullet)$ where $V$ is a $\R$-local system on $B$ and $\mathcal{F}^\bullet$ is a filtration by holomorphic subbundles of $V_\C \otimes \mathcal{O}_X$ such that \begin{itemize}
        \item Over every point $x \in B$, $\mathcal{F}_{x}^{\bullet}$ is the Hodge filtration for some $\R$-Hodge structure of weight $n$ on $V_x$ ;
        \item For the flat bundle $(\mathcal{H}_\C, \nabla)$ associated to $V_\C$, one has $\nabla(\mathcal{F}^p) \subset \mathcal{F}^{p-1} \otimes \Omega^1_X$ for all $p \in \{1, \dots, n \}$.
    \end{itemize}
    This last condition is usually called \textit{Griffiths transversality}.
\end{df}

If $(V, \mathcal{F})$ is a VHS of weight 2 on $B$, then one can consider the $\mathcal{C}^{\infty}$ decomposition $\mathcal{H}_\C = \mathcal{H}^{2,0} \oplus \mathcal{H}^{1,1} \oplus \mathcal{H}^{0,2}$ where $\mathcal{H}^{2,0} := \mathcal{F}^2$, $\mathcal{H}^{0,2} = \overline{\mathcal{H}^{2,0}}$ and $\mathcal{H}^{1,1} := \mathcal{F}^1 \cap \overline{F^1}$. 

\begin{df}
    A VHS of \textit{$\mathrm{K3}$-type} on $B$ is a triple $(V, \mathcal{F}^\bullet, q)$, where $(V, \mathcal{F}^\bullet)$ is a $\R$-VHS of weight $2$ on $B$ with $h^{2,0} = 1$ and $q$ is a flat metric on the flat bundle $(\mathcal{H}, \nabla)$ associated to $V$ such that 
    \[ h := \begin{cases}
        q(x, \overline{y}), & \text{for } x,y \in \mathcal{H}^{2,0} \oplus \mathcal{H}^{0,2} \\
        -q(x, \overline{y}), & \text{for } x,y \in \mathcal{H}^{1,1} \\
        0 & \text{otherwise}
    \end{cases} \]
    is such that
    \begin{itemize}
        \item $h$ is positive definite on $\mathcal{H}^{2,0} + \mathcal{H}^{0,2}$ ;
        \item $h$ is of signature $(h^{1,1} -1, 1)$ on $\mathcal{H}^{1,1}$.
    \end{itemize}
\end{df}

\begin{rmk}\label{rmk VHS polarisées}
    The metric $q$ is not a polarization in the usual sense (for example in the sense of \cite{Carlson_Müller-Stach_Peters_2017}) since $h$ it is not positive definite on $\mathcal{H}^{1,1}$. One way to make it a polarization would be to choose a flat and real section $s$ of $\mathcal{H}^{1,1}$ such that $q(s,s) >0$ and to consider the induced VHS on $s^\perp \subset \mathcal{H}$. Such a choice is not always possible. In the geometric situation, such a choice can be made by choosing a relative Kähler class, if any exist.
\end{rmk}

Let $H$ be a real vector space endowed with a non-degenerate symmetric bilinear form $q$ of signature $(3,p)$. 

\begin{df}\label{def period domain}
    The \textit{period domain} of $\mathrm{K3}$-type associated to $(H,q)$ is 
    \[ D := \{\sigma \in \mathbb{P}(H_{\mathbb{C}}) \mid q(\sigma, \sigma) = 0, q(\sigma, \overline{\sigma}) > 0 \}.\]
\end{df}

As explained in Section \ref{section familles}, if $B$ is endowed with a VHS of $\mathrm{K3}$-type and if there exists an isomorphism $V \simeq \underline{H}$, then one obtains a \textit{period map} $\mathcal{P} : B \rightarrow D$ that is holomorphic.

The set $D$ can be seen as an open subset set of the quadric
\[\check{D} := \{\sigma \in \mathbb{P}(H_{\mathbb{C}}) \mid q(\sigma, \sigma) = 0\},\] 
which is smooth because $q$ is non-degenerate. Thus, $D$ inherits the structure of a complex manifold. Every element of $D$ corresponds to a real Hodge structure of $\mathrm{K3}$-type on $H$. The Hodge structure associated to a point $l \in D$ is given by $H^{2,0}_l = l$, $H^{0,2}_l = \overline{l}$ and $H^{1,1}_l = (l \oplus\overline{l})^\perp$. 

There is a natural bijection between $D$ and the set
\[ \{P \in Gr^{o}(2, H) \mid q|_{P} \text{ is positive definite} \}\]
of oriented real positive 2-planes in $H$. It is given by the following map : to $l \in D$ we associate the real points of the complex 2-plane $l \oplus \overline{l}$ together with the induced orientation coming from $H_{\mathbb{C}}$. One can show that these points form a real 2-plane inside $H$ and that this map is a bijection (see, for example, \cite[p.101]{Huybrechts_2016}).

Notice that the group $G$ of linear transformations of $H$ that preserve $q$, which is isomorphic to $\mathrm{O}(3,p)$, acts by biholomorphisms on $D$. One can show that the action of its identity component is transitive (see, for example, \cite[Section 16.4]{FHW}). 

Consider the constant vector bundle $\mathcal{H}$ of fiber $H$ on $D$. This bundle admits a tautological decreasing filtration $\mathcal{F}^\bullet$ that is given fiberwise by the Hodge filtration corresponding to the point. This filtration is holomorphic. Indeed, $\mathcal{F}^2 \simeq \mathcal{O}_{\mathbb{P}(H_\C)}(-1)|_D$ and $\mathcal{F}^1 = (\mathcal{F}^2)^\perp$, where the orthogonal is taken with respect to the complexification of the form $q$. Notice that $G$ preserves this filtration. 

It is known that $T_{\mathbb{P}(V)} = \mathrm{Hom}(\mathcal{O}_{\mathbb{P}(V)}(-1), V/\mathcal{O}_{\mathbb{P}(V)}(-1))$, see for example \cite[p.228]{VoisinHodge}. Under this identification, at a point $p \in D$ one has
\[T_{D,p} = \{ u \in  \mathrm{Hom}(\mathcal{F}^2_p, V/\mathcal{F}^2_p) \mid q(u(\mathcal{F}^2_p), \mathcal{F}^2_p) = 0 \} = \mathrm{Hom}(\mathcal{F}^2_p, \mathcal{F}^1_p/\mathcal{F}^2_p),\]
and thus $T_{D} = \mathrm{Hom}(\mathcal{F}^2, \mathcal{F}^1/\mathcal{F}^2)$. This observation has the following consequence.

\begin{prop}\label{prop VHS sur le domaine}
    The data of $(\underline{H}, \mathcal{F}, q)$ forms a VHS of $\mathrm{K3}$-type on $D$. 
\end{prop}

\begin{proof}
    It suffices to check that Griffiths transversality holds, which in this case simplifies to $\nabla(\mathcal{F}^2) \subset \mathcal{F}^1 \otimes \mathcal{O}_D$. The connection $\nabla$ on $\mathcal{H}$ is simply the action of vector fields on sections on $\mathcal{H}$ seen as functions from $D$ to $H$. Therefore, Griffiths transversality follows from the discussion above and the description of the identification $T_{\mathbb{P}(V)} = \mathrm{Hom}(\mathcal{O}_{\mathbb{P}(V)}(-1), V/\mathcal{O}_{\mathbb{P}(V)}(-1))$ given in \cite{VoisinHodge}. 
\end{proof}

Since the pullback along a holomorphic map of a VHS is still a VHS, one obtains the following. 

\begin{cor}
    Every holomorphic map $B \rightarrow D$ is the period map of a VHS of $\mathrm{K3}$-type on $B$. 
\end{cor}

\begin{rmk}
    In the classical theory of period domains corresponding to polarized Hodge structures, the tautological filtration on the domain is not always a VHS, i.e. Griffiths transversality condition might fail. Actually, Griffiths transversality condition is satisfied if and only if the domain in question is a bounded symmetric domain (see for example \cite[Remark 12.5.5]{Carlson_Müller-Stach_Peters_2017}). In our case, one can show that $D$ is diffeomorphic to $\SO^o(3,p)/\SO(2) \times \SO^o(1,p)$, which is a semi-hermitian symmetric manifold in the sense of \cite{o1983semi}. With this formalism, the proof of Proposition \ref{prop VHS sur le domaine} is the same as the classical one for bounded symmetric domains. 
\end{rmk}

\subsection{The Griffiths-Schmid metric}\label{section metric GS}

If $(V, \mathcal{F}^\bullet, q)$ is a VHS of $\mathrm{K3}$-type on $B$, $\mathcal{F}^2$ is a line bundle. As for VHS induced by families, we will call this bundle the \textit{Hodge bundle}, and we will denote it by $\mathcal{L}$. This bundle comes endowed with the restriction of $h$.

\begin{df}
    On $D$, we define the \textit{Griffiths-Schmid metric} to be $\omega_D := i \Theta(\mathcal{L},h)$, where $\mathcal{L}$ is the Hodge bundle of the tautological VHS on $D$. 
\end{df} 

\begin{rmk}
    If $\pi : \mathcal{X} \rightarrow B$ is a family of \HK manifolds, then by construction the curvature of the Hodge bundle coincides with the metric on $B$ defined by pullback of the local period maps induced by $\pi$. 
\end{rmk}

\begin{prop}\label{metric de GS}
    The Griffiths-Schmid metric $\omega_D$ on $D$ associated to the tautological VHS coincides, up to a factor 2, with the metric induced by $h$ on $T_{D} = \mathrm{Hom}(\mathcal{F}^2, \mathcal{F}^1/\mathcal{F}^2)$. In particular, $\omega_D$ is an everywhere non-degenerate pseudo-Kähler metric on $D$ that is invariant under the action of $G$ and of signature $(p,1)$. 
\end{prop}

\begin{proof}
    The fact that $\omega$ is pseudo-Kähler and invariant under the action of $G$ follows from the fact that it is defined as the curvature form of a $G$-invariant metric on $\mathcal{L}$. Without loss of generality, we can assume that $H = \R^{3+p}$ endowed with the standard metric of signature $(3,p)$. Notice that $\omega_D$ is the restriction to $D$ of a form $\omega_U$ defined on the open subset $U$ of positive lines in $\mathbb{P}^{p+2}$. Let $V$ be the standard open coordinate set of $\mathbb{P}^{p+2}$ given by the non-vanishing of the first coordinate. Let $l$ be the following function on $U \cap V$ defined by 
    \[l(z) = 1 + \sum_{i = 1}^2 |z_{i}|^2 - \sum_{j = 3}^{p+2} |z_j|^2.\]
    In these coordinates, $\omega_U$ has the following expression. 
    \begin{align*}
        \omega_U &= -i\partial\overline{\partial}\log(l) \\
        &= i \left(\frac{\partial l}{l} \wedge \frac{\overline{\partial}l}{l} - \frac{\ddbar l}{l}\right)\\
        &= i\frac{\left(\sum_{i = 1}^2 \overline{z}_idz_i - \sum_{j = 3}^{p+2} \overline{z}_jdz_j \right) \wedge \left(\sum_{i = 1}^2 z_id\overline{z}_i - \sum_{j = 3}^{p+2} z_jd\overline{z}_j\right)}{l^2} \\
        & \quad - i\frac{\sum_{i = 1}^2 dz_i \wedge d\overline{z}_i - \sum_{j = 3}^{p+2} dz_j \wedge d\overline{z}_j}{l}. \\
    \end{align*}
    Let $\Omega_U :=(h_{i,j})_{1 \leqslant i,j \leqslant p}$ be the matrix such that $\omega_U = \sum_{i,j=1}^{p+2} h_{i,j} \frac{i}{2}dz_i\wedge d\overline{z}_j$. Consider the point $o := [1:i:0: \cdots :0] \in D$. Notice that the tangent space $T_{D,o}$ is the kernel of $dz_1$. The previous calculation shows that $\Omega(o) = \mathrm{diag}(-\frac{1}{2},-1,1,...,1)$, and thus if one chooses $\frac{\partial}{\partial z_2}|_o, \dots,\frac{\partial}{\partial z_{p+2}}|_o$ as a basis of $T_{D,o}$ the matrix of $\omega_D$ at $o$ is $\mathrm{diag}(-1,1,...,1)$. The Hodge structure corresponding to $o$ is such that $H^{1,1}_o = \mathrm{Span}(e_3, \dots, e_{p+3})$ where $e_1, \dots, e_{p+3}$ is the standard basis of $\R^{p+3}$. Using the description of \cite[p.228]{VoisinHodge}, one concludes that $\frac{1}{2}\omega_D$ coincides with the metric induced by $h$ on $T_{D} = \mathrm{Hom}(\mathcal{F}^2, \mathcal{F}^1/\mathcal{F}^2)$ at the point $o$. Since both metrics are $G$-invariant and $G$ acts transitively on $D$, the result holds at every point of $D$. 
\end{proof}

We introduce the following definition that will be used in Section \ref{section 6}. 

\begin{df}\label{def domaine Omega}
    We define the domain 
    \[ \Omega := \{ l \in \mathbb{P}(H_\C) \mid \text{$h$ is positive on $l$} \}\]
    and the metric $\omega := i \Theta(\mathcal{O}_{\mathbb{P}(H_\C)}(-1)|_\Omega, h)$ on $\Omega$.
\end{df}

The previous computations show that $\omega$ is an everywhere non-degenerate pseudo-Kähler metric of signature $(p,2)$, that is $\mathrm{PU}(3,p)$-invariant, and that restricts to $\omega_D$ on $D$. 

\begin{rmk}
    The metric $\omega_D$ has been considered in \cite[Appendix]{DEEV}.
\end{rmk} 

\begin{ex}[Polarized domains]\label{exemple domaines polarisés}
    Let $\alpha \in V$ be such that $q(\alpha, \alpha) > 0$, and consider 
    \[\alpha^\perp := \{ l \in D \mid q(l, \alpha) = 0 \}. \]
    One easily sees that $\alpha^\perp$ is a submanifold homogeneous under the action of the subgroup of $G$ formed by elements preserving $\alpha$, which is isomorphic to $\mathrm{O}(2,p)$. Furthermore, $\alpha^\perp$ has two connected components, each of which is diffeomorphic to $\SO^0(2,p)/\SO(2) \times \SO(p)$. In particular, each component of $\alpha^\perp$ is a bounded symmetric domain and thus carries a natural positive definite metric called the Bergman metric (see, for example, \cite{Helgason}). In view of Remark \ref{rmk VHS polarisées}, $\alpha^\perp$ parametrizes polarized $\R$-VHS and thus one can consider the usual Griffiths-Schmid metric constructed in \cite{GriffithsSchmid69} (see also \cite{Carlson_Müller-Stach_Peters_2017}). Because of the description given in Proposition \ref{metric de GS}, these metrics both coincide, up to a positive constant, to the restriction of $\omega_D$ to $\alpha^\perp$ and $\alpha^\perp$ is embedded as a totally geodesic submanifold of $D$.
\end{ex}

\begin{ex}[Twistor lines]\label{exemple twistor}
    Let $W \subset V$ be a positive 3-plane, i.e. a 3-plane on which $q$ restricts as a positive definite form. Consider 
    \[T_W := \{ l \in D \mid l \subset W_\C \}.\]
    Using the description of elements of $D$ in terms of positive 2-planes, one sees that $T_W$ is an orbit of a subgroup of $G$ isomorphic to $\mathrm{O}(3)$. In particular, $T_W$ is a smooth submanifold embedded in a totally geodesic way inside $D$. Since $W$ is positive, the condition $q(l, \overline{l}) > 0$ is automatic. The manifold $T_W$ is therefore a smooth conic in $\mathbb{P}^2$ and is thus isomorphic to $\mathbb{P}^1$ via an embedding $\iota : \mathbb{P}^1 \rightarrow D$ of degree 2. One has $\iota^*\mathcal{L} \simeq \mathcal{O}_{\mathbb{P}^1}(-2)$ and $\iota^*h_{BB} = h_{FS}^2$ where $h_{FS}$ is a metric on $\mathcal{O}_{\mathbb{P}^1}(-1)$ induced by a positive definite hermitian form on $\C^2$. In particular, the induced metric $\iota^*\omega_D$ is a negative multiple of a Fubini-Study metric on $\mathbb{P}^1$. 
    
    Let $X$ be a \HK manifold of dimension $2n$, and let $\alpha \in H^{1,1}(X, \R)$ be a Kähler class. By Yau's theorem, there exists a (unique) Ricci flat metric $\omega \in \alpha$. It is known that the holonomy group of this metric is $\mathrm{Sp}(n)$. In particular, it is a \HK metric, and thus, $X$ is a member of a smooth family of \HK manifolds $\mathcal{T}_{X, \alpha} \rightarrow \mathbb{P}^1$ that is everywhere of maximal variation called the \textit{twistor family}. The choice of a marking $\varphi : H^2(X, \Z) \rightarrow \Lambda$ induces a marking of $\mathcal{T}_{X, \alpha} \rightarrow \mathbb{P}^1$. By construction, the image of the period map associated to this marked family is $T_{W}$ where $W := \varphi((H^{2,0}(X) + H^{0,2}(X))_\R \oplus \R \cdot \alpha)$. We refer to \cite[Section 25]{GrossJoyceHuybrechts} for more details. 
\end{ex}

\begin{rmk}
    Consider $\alpha \in V$ as in Example \ref{exemple domaines polarisés} above. Let $p \in D$ be a point such that $\alpha \in H^{1,1}_p$ and let $P_p$ be the positive 2-plane associated to $p$. Let $W := P_p \oplus \R \cdot \alpha$. The twistor line $T_W$ and the polarized subdomain $\alpha^\perp$ both contain $\alpha$ and one has $T_{D, p} = T_{\alpha^\perp, p} \oplus^\perp T_{T_W, p}$.
\end{rmk}

\subsection{Domains of dimension 2}\label{section domaine dim 2}

In this section we study the case of period domains of $\mathrm{K3}$-type that are of dimension 2. These domains can be realized as an explicit open and dense subset of $\mathbb{P}^1 \times \mathbb{P}^1$. Using this description, constructions made in the previous section can be computed explicitly. It is this explicit description that will allow us to construct positive holomorphic disks needed in the proof of Theorem \ref{thm dist koba intro}.

\begin{rmk}\label{rmq domaines dim 2}
    Let $D$ be a period domain of $\mathrm{K3}$-type associated to a pair $(H,q)$. If $D$ is of dimension at least 2, then every point $p \in D$ is contained inside a 2-dimensional subdomain of $D$ of $\mathrm{K3}$-type. Indeed, if $P$ is the positive 2-plane inside $H$ associated to the point $p$ and $Q \subset P^\perp$ is a 2-plane of signature $(1,1)$ with respect to $q$, then $D \cap (P \oplus Q)$ is a 2-dimensional subdomain of $D$ containing $p$. 
\end{rmk}

Consider $H = \R^4$ endowed with the form $q$ which is the standard quadratic form of signature $(3,1)$. Let $D_2$ be the domain of $\mathrm{K3}$-type associated to $(H,q)$. This domain is an open subset inside the quadric 
\[\check{D}_{2} = \{ x^2 + y^2 + z^2 - t^2 = 0\} \subset \mathbb{P}^3.\]
Any such quadric is isomorphic to $\mathbb{P}^1 \times \mathbb{P}^1$. Here we can choose the isomorphism given by 
\[\iota : ([x_0: x_1], [y_0:y_1]) \mapsto [x_0y_0 + x_1y_1:i(x_1y_1 - x_0y_0):x_1y_0 -x_0y_1:x_1y_0 + x_0y_1].\]
On $\check{D}_2$, consider the anti-holomorphic involution $\tau$ given by 
\[\tau([x:y:z:t]) = [\overline{x}:\overline{y}:\overline{z}:\overline{t}].\] 

\begin{lem}\label{bord de D_2}
    The set $ B := \check{D}_2 \setminus D_2$ corresponds to the set of fixed points of $\tau$. 
\end{lem}

\begin{proof}
    Let us show that $B$ is the set of points in $\check{D_2}$ that admit real homogeneous coordinates. Start by noticing that if $[x:y:z:t] \in \check{D_2}$, then 
    \[|x|^2 + |y|^2 + |z|^2 - |t|^2 \geqslant |x^2 + y^2 + z^2| - |t|^2 = 0.\]
    Therefore, the points $[x:y:z:t] \in \check{D}_2 \setminus D_2$ are the ones satisfying $|x|^2 + |y|^2 + |z|^2 = |x^2 + y^2 + z^2|$, which amounts to the fact that there exists $\alpha$ with norm 1 such that $x = \alpha x'$, $y = \alpha y'$, $z = \alpha z'$ and $x',y',z' \in \mathbb{R}$. This implies that $t$ is also of the form $t = \alpha t'$, which concludes.
\end{proof}

The pullback of $\tau$ to $\mathbb{P}^1 \times \mathbb{P}^1$, that we still denote by $\tau$, is given by 
\[\tau([x_0: x_1], [y_0:y_1]) = ([\overline{y_1}: \overline{y_0}], [\overline{x_1}: \overline{x_0}]).\]

\begin{rmk}
    The fixed points of $\tau$ are of the form $([x_0:x_1], [\overline{x_1}:\overline{x_0}])$ for $[x_0:x_1] \in \mathbb{P}^1$. In particular, $B$ is a 2-dimensional compact sphere inside $\check{D}_2$. As such, it has a fundamental class $[B] \in H^2(\check{D}_2, \Z)$. Every line bundle of $\mathbb{P}^1 \times \mathbb{P}^1$ is of the form $\mathcal{O}(a,b) := \pr_1^*\mathcal{O}_{\mathbb{P}^1}(a) \otimes \pr_2^*\mathcal{O}_{\mathbb{P}^1}(b)$ and any smooth curve in $|\mathcal{O}(a,b)|$ is of genus $(a-1)(b-1)$. For the complex structure $\mathbb{P}^1 \times \overline{\mathbb{P}^1}$, $B$ is a divisor in $\mathcal{O}(1,1)$. One therefore has $[B] = c_1(\mathcal{O}(1,-1))$ in $H^{1,1}(\mathbb{P}^1 \times \mathbb{P}^1, \Z)$. In particular, every compact curve in $D_2$ belongs to a linear system $|\mathcal{O}(n,n)|$ for $n \geqslant 0$ and thus every smooth rational curve inside $D_2$ (such as a twistor line) belongs to $|\mathcal{O}(1,1)|$.
\end{rmk}

Let us now give an explicit form for $\omega_{D}$ on $D_2$. Consider the chart $U := \C \times \C \subset \mathbb{P}^1 \times \mathbb{P}^1$ given by $(x,y) \mapsto ([x:1], [y:1])$. One has 
\[ (\C \times \C) \cap D_2 = \{(x,y) \mid x\overline{y} \neq 1\}.\]

\begin{prop}\label{metric dim 2}
    In the chart $(\C \times \C) \cap D_2$, the matrix form $\Omega$ of $\omega_D$ is given by 
    \[\Omega = \begin{pmatrix}
        0 & \frac{1}{(x\overline{y}-1)^2}\\
        \frac{1}{(y\overline{x}-1)^2} & 0
    \end{pmatrix}.\]
\end{prop}

\begin{proof}
    Let $p = (x,y) \in (\C \times \C) \cap D_2$ and $l := \iota(p)$. 
    Using the identification $T_{\mathbb{P}^3,l} = \Hom(l, \C^4/l)$, one has $\iota_* \frac{\partial}{\partial x} = \varphi$ and $\iota_* \frac{\partial}{\partial y} = \psi$ where $\varphi$ and $\psi$ are such that $\varphi(u) = v_x$ and $\psi(u) = v_y$
    where 
    \begin{gather*}
        u = (xy+1, i(1-xy), y-x, x+y) \\
        v_x = (y, -iy, -1, 1) \\
        v_y = (x, -ix, 1, 1). 
    \end{gather*}
    Consider the Hodge structure on $\C^4$ corresponding to $l$, and decompose $v_x = \lambda_x u + \alpha_x$ and $v_y = \lambda_y u + \alpha_y$ where $\alpha_x, \alpha_y \in H^{1,1}$. By Proposition \ref{metric de GS}, one has 
    \begin{gather*}
        h(\iota_* \frac{\partial}{\partial x}, \iota_* \frac{\partial}{\partial x}) = -\frac{q(\alpha_x, \overline{\alpha_x})}{q(u, \overline{u})} = -\frac{q(v_x, \overline{v_x})q(u, \overline{u}) - |q(v_x, \overline{u})|^2}{q(u, \overline{u})^2} = 0,\\
        h(\iota_* \frac{\partial}{\partial y}, \iota_* \frac{\partial}{\partial y}) = -\frac{q(\alpha_y, \overline{\alpha_y})}{q(u, \overline{u})} = -\frac{q(v_y, \overline{v_y})q(u, \overline{u}) - |q(v_y, \overline{u})|^2}{q(u, \overline{u})^2} = 0, \\
        h(\iota_* \frac{\partial}{\partial x}, \iota_* \frac{\partial}{\partial y}) = -\frac{q(\alpha_x, \overline{\alpha_y})}{q(u, \overline{u})} = -\frac{q(v_x, \overline{v_y})q(u, \overline{u}) - q(v_x, \overline{u})q(u, \overline{v_y})}{q(u, \overline{u})^2} = \frac{1}{(x\overline{y}-1)^2}. \\
    \end{gather*}    
\end{proof}

\begin{cor}\label{coro involution isométrie}
    The involution $i : \mathbb{P}^1 \times \mathbb{P}^1 \rightarrow \mathbb{P}^1 \times \mathbb{P}^1, (p,q) \mapsto (q,p)$ maps $D_2$ onto itself and is an isometry for $\omega$.  
\end{cor}

\begin{proof}
    The first assertion follows from Lemma \ref{bord de D_2}. The second follows form the fact that the continuous function 
    \[T_{D_2} \times_{D_2} T_{D_2} \rightarrow \C, \quad (u,v) \mapsto h(u,v) - h(i_*u, i_*v)\]
    vanishes on the dense open subset $T_U \times_U T_U$ by Proposition \ref{metric dim 2}.
\end{proof}

Let us also mention the following observation that follows directly from the above study.

\begin{cor}\label{coro droites horizontales}
    The fibers of the two natural projections $\mathrm{pr}_1$ and $\mathrm{pr}_2$ of $D$ onto $\mathbb{P}^1$ are isomorphic to $\C$ and are isotropic for $\omega_D$.
\end{cor}

It is known that $\SO^o(3,1)$ is isomorphic to $\mathrm{PSL}_2(\C)$, see for example \cite[p.565]{Helgason}. Using the isomorphism described in \cite[p.565]{Helgason}, the action of $G$ on $D_2$ is described as follows. 

\begin{prop}\label{description action D2}
    Let $p :=([x_0:x_1], [y_0:y_1]) \in D_2$ and $A := \begin{pmatrix}
        a & b \\
        c & d
    \end{pmatrix} \in \mathrm{SL}_2(\C)$. Then one has 
    \[ A \cdot p = (A \cdot [x_0:x_1], \widetilde{A} \cdot [y_0:y_1]), \]
    where 
    \[\widetilde{A} = \begin{pmatrix}
        \overline{d} & \overline{c} \\
        \overline{b} & \overline{a}
    \end{pmatrix}.\]
\end{prop}

\begin{proof}
    To each element $X := (x_1,x_2,x_3,x_4) \in \C^4$, one associates the following matrix
    \[H(X) := \begin{pmatrix}
        i(x_4-x_3) & -x_2 + ix_1 \\
        x_2 + ix_1 & i(x_4 + x_3)
    \end{pmatrix}.\]
    By \cite[p.565]{Helgason}, an element $A \in \mathrm{SL}_2(\C)$ acts on $X$ by 
    \[A \cdot X := AH\overline{A}^t.\]
    For an element $p :=([x_0:x_1], [y_0:y_1]) \in \mathbb{P}^1 \times \mathbb{P}^1$, the associated matrix for $\iota(p)$ is 
    \[H(\iota(p)) = 2i \begin{pmatrix}
        x_0y_1 & x_0y_0 \\
        x_1y_1 & x_1y_0
    \end{pmatrix}.\]
    The result then follows by a direct computation. 
\end{proof}

\section{Families with prescribed period map}\label{section lifting curves}

Let $\Lambda$ be a fixed lattice for which $\mathcal{M}_\Lambda$ is not empty. Let $\mathcal{M}^0_\Lambda$ be a connected component of $\mathcal{M}_\Lambda$. The goal of this section is to prove the following. 

\begin{thm}\label{thm lift}
    Let $C$ be a smooth (non-necessarily compact) connected curve and let $x_0 \in C$. Let $f : C \rightarrow D_\Lambda$ be holomorphic map and $(X_0,\varphi_0) \in \mathcal{M}_\Lambda^0$ be a marked \HK manifold such that $\mathcal{P}(X_0, \varphi_0) = f(x_0)$. There exists a family $\mathcal{X} \rightarrow C$ and a marking $\mathrm{R}^2\pi_*\underline{\Z} \rightarrow \underline{\Lambda}$ such that \begin{itemize}
        \item $f$ is the period map induced by this marked family ;
        \item $(X_0, \varphi_0) \simeq (X_{x_0}, \varphi_{x_0})$.
    \end{itemize}
\end{thm}

Notice that, together with the computations of Section \ref{section domaine dim 2}, we obtain the following. 

\begin{cor}\label{coro familles plates}
    Let $X$ be a \HK manifold. Then there exists a family $\pi : \mathcal{X} \rightarrow \C$ with everywhere maximal variation such that : \begin{itemize}
        \item the Hodge bundle $(\pi_*\Omega^2_{\mathcal{X}/\C}, h_{BB})$ is flat ;
        \item $X \simeq \pi^{-1}(0)$.
    \end{itemize}
\end{cor}

\begin{proof}
    Choose a $\Lambda$-marking on $X$ and consider the associated period point $p \in D_\Lambda$. By Remark \ref{rmq domaines dim 2}, this point is contained in a 2-dimensional subdomain of $D_\Lambda$. By Corollary \ref{coro droites horizontales}, there exists an embedding $f : \C \rightarrow D_\Lambda$ whose image is isotropic and such that $f(0) = p$. Applying Theorem \ref{thm lift}, we obtain a family as desired. 
\end{proof}

\subsection{Global Torelli theorem and MBM classes}\label{section globl torelli}

We start by recalling the main tools that we will use to prove Theorem \ref{thm lift}, namely, MBM classes introduced by Amerik-Verbitsky in \cite{AmerikVerbitsky} and the description of the fibers of the period map given by Markman in \cite{MarkmanSurvey}. We claim no originality for the material presented in this section. 

Let us start by recalling the following celebrated theorem due to Huybrechts. 

\begin{thm}[\cite{HuybrechtsBasicResults}]\label{thm surjectivité}
    The map $\mathcal{P}$ is surjective. 
\end{thm}

We now recall a few facts about the so-called MBM classes introduced in \cite{AmerikVerbitsky}. We also refer the reader to \cite{AmerikVerbitskySurvey} for more details. 

Two \HK manifolds $X$ and $X'$ are said to be deformation equivalent if there exists a family $\mathcal{X} \rightarrow B$ of \HK manifolds, where $B$ is connected complex space, and two points $p,q \in B$ such that $X \simeq X_p$ and $X' \simeq X_q$. Notice that the choice of a path from $p$ to $q$ provides an isomorphism between $H^2(X, \Z)$ and $H^2(X', \Z)$. Such an isomorphism will be called a monodromy transformation. If $p=q$, the set of monodromy transformations forms a subgroup of $\mathrm{O}(H^2(X, \Z))$ that we will denote by $\mathrm{Mon}(X)$. We will denote by  $\mathrm{Mon_{Hdg}}(X)$ the subgroup of  $\mathrm{Mon}(X)$ of elements preserving the Hodge structure on $H^2(X, \Z)$.

Let $X$ be a \HK manifold. The Beauville-Bogomolov form gives a canonical isomorphism between $H^2(X, \Q) \simeq H_2(X,\Q)$. Moreover, since the Fujiki constant is invariant under deformation, this identification is preserved under deformation of $X$. From now on, we will freely identify $H^2(X, \Q)$ and $H_2(X,\Q)$ and consider the Beauville-Bogomolov form on $H_2(X, \Q)$.

\begin{df}[{\cite[Theorem 1.17, Theorem 5.10]{AmerikVerbitsky}}]\label{def classes MBM}
    A rational homology class $\alpha \in H_2(X, \Q)$ of negative $q_X$ square is called \textit{monodromy birationally minimal}, or \textit{MBM} for short, if $\alpha \in H^{1,1}(X)$ and there exists a \HK manifold $X'$ which is deformation equivalent to $X$ such that under a monodromy operator $\Phi : H^2(X, \Q) \xrightarrow{\sim} H^2(X', \Q)$ one has
    \begin{itemize}
        \item $\mathrm{Pic}(X')_\Q = <\Phi(\alpha)>$ ;
        \item there exists $k \in \Z$ such that $k \cdot \Phi(\alpha)$ is the class of a rational curve on $X'$.
    \end{itemize}
    We denote by $\mathrm{MBM}(X) \subset H^{1,1}(X, \Q)$ the set MBM classes of $X$. 
\end{df} 

Let us start by noticing the following.

\begin{lem}\label{lem classes MBM}
    Let $(X, \varphi), (X', \psi) \in \mathcal{M}^0_\Lambda$ be such that $\mathcal{P}(X, \varphi) = \mathcal{P}(X', \psi)$. Then $\psi^{-1} \circ \varphi : H^2(X, \Q) \rightarrow H^2(X', \Q)$ induces a bijection between $\mathrm{MBM}(X)$ and $\mathrm{MBM}(X')$. 
\end{lem}

\begin{proof}
    Since $(X, \varphi)$ and $(X', \psi)$ have the same period, $\psi^{-1} \circ \varphi$ induces a bijection between classes of type $(1,1)$. Moreover, since $(X, \varphi)$ and $(X', \psi)$ are in the same connected component, $\psi^{-1} \circ \varphi$ is a monodromy operator. The result therefore follows directly from Definition \ref{def classes MBM} (see also \cite[Theorem 1.17]{AmerikVerbitsky}).
\end{proof}

Let $\mathcal{C} \subset H^{1,1}(X, \R)$ be the set of positive $q_X$ square elements of $H^{1,1}(X, \R)$. This set has two connected components, with one containing all Kähler classes. We denote by $\mathrm{Pos}(X)$ this component and by $\Kah(X)$ the set of Kähler classes. The interest for MBM classes stems from the following theorem. 

\begin{thm}[{\cite[Theorem 1.19]{AmerikVerbitsky}}]\label{thm AV decomp cone}
    The connected components of 
    \[ \mathrm{Pos}(X) \setminus \bigcup_{\alpha \in \mathrm{MBM}(X)} \alpha^\perp\] 
    are the sets of the form $g(u^*(\Kah(X)))$ for $g \in \mathrm{Mon_{Hdg}}(X)$ and  $u : X \dashrightarrow X'$ a bimeromorphic map with $X'$ a \HK manifold. Such a set is called a \textit{Kähler-type chamber} of $X$. 
\end{thm}

Consider the vector bundle $\mathcal{H}^{1,1}$ over $D_\Lambda$ and let $\mathcal{C}$ be the subset of $\mathcal{H}^{1,1}$ formed of elements with positive square. Over each point of $D_\Lambda$, $\mathcal{C}$ has two connected components. Since $D_\Lambda$ is simply connected, we can continuously choose one of these connected component that we call $\mathcal{C}^+$. Furthermore, by \cite[Section 4]{MarkmanSurvey}, we can choose this component so that for every pair $(X, \varphi) \in \mathcal{M}^0_\Lambda$, one has $\varphi(\mathrm{Pos}(X)) = \mathcal{C}^+_{\mathcal{P}(X, \varphi)}$. 

\begin{rmk}\label{rmk action sur le cone}
    The group $\mathrm{O}(\Lambda_\R)$ acts on $\mathcal{C}$, and the subgroup $\SO^0(\Lambda_\R)$ preserves the choice of a connected component $\mathcal{C}^+$.
\end{rmk}

Let $p \in D_\Lambda$. Lemma \ref{lem classes MBM} implies that for all $(X, \varphi) \in \mathcal{P}^{-1}(p)$, the sets $\varphi(\mathrm{MBM}(X))$ are the same. We are thus led to define the following. 

\begin{df}\label{def MBM domaine}
    Let $\alpha \in \Lambda_\Q$. We will say that $\alpha$ is \textit{MBM} at $p$ if there exists a marked pair $(X, \varphi) \in \mathcal{M}^0_\Lambda$ such that $\mathcal{P}(X, \varphi) = p$ and $\varphi^{-1}(\alpha)$ is MBM. The set of MBM classes at $p$ will be denoted by $\mathrm{MBM}(p)$. We will denote by $\mathrm{MBM}(\mathcal{M}^0_\Lambda) := \bigcup_{p \in D_\Lambda}\mathrm{MBM}(p)$ the set of all MBM classes with respect to $\mathcal{M}^0_\Lambda$.
\end{df}

\begin{rmk}\label{remarque classes MBM}
    If $p \in D_\Lambda$ and $\alpha \in \Lambda_\Q$ is such that $q(\alpha, \alpha) < 0$, then $\alpha \in \mathrm{MBM}(p)$ if and only if the following holds : \begin{enumerate}
        \item $\alpha$ is of type $(1,1)$ for the Hodge structure corresponding to $p$ ;
        \item \label{item 2} there exists $p' \in D_\Lambda$ such that $H^{1,1}_{p'} \cap \Lambda_\Q = <\alpha>$ and there exists $(X, \varphi) \in \mathcal{P}^{-1}(p')$ such that $\varphi^{-1}(\alpha)$ is a rational multiple of the class of a rational curve in $X$. 
    \end{enumerate}
    In particular, the set $\mathrm{MBM}(\mathcal{M}^0_\Lambda)$ is exactly the set of classes satisfying condition \ref{item 2}. If $\alpha \in \mathrm{MBM}(\mathcal{M}^0_\Lambda)$ then the points $p \in D_\Lambda$ such that $\alpha \in \mathrm{MBM}(p)$ is a divisor, namely, $\alpha^\perp$.
\end{rmk} 

As in Theorem \ref{thm AV decomp cone}, we can consider the decomposition 
\[ \mathcal{C}^+_p \setminus \bigcup_{\alpha \in \mathrm{MBM}(p)} \alpha^\perp.\]
We will call the connected components of this decomposition the Kähler-type chambers of $\mathcal{C}^+_p$, and we will denote the set of such Kähler-type chambers by $\mathrm{KT}(p)$. 

\begin{thm}[{\cite[Theorem 5.16]{MarkmanSurvey}, \cite{VerbitskyTorelli}}]\label{thm Torelli Global}
    The map 
    \[ \mathcal{P}^{-1}(p) \rightarrow \mathrm{KT}(p), \quad (X, \varphi) \mapsto \varphi(\Kah(X))\]
    is a bijection. In particular, $\mathcal{P}$ is injective over the complement of the divisors of the form $\alpha^\perp$ for $\alpha \in \mathrm{MBM}(\mathcal{M}^0_\Lambda)$. 
\end{thm}

\subsection{Proof of Theorem \ref{thm lift}}

As explained in Section \ref{Section intro}, we will follow the strategy of \cite[Theorem 4.4]{GrebSchwald}. We start by showing two lemmas that will be needed in the proof of Theorem \ref{thm lift}.

\begin{lem}\label{lem domaine de Mumford Tate}
    Let $N$ be a subspace of $\Lambda_\mathbb{Q}$. We define 
    \[ D_N := \{ p \in D_\Lambda \mid N \subset H^{1,1}_p \cap \Lambda_\mathbb{Q} \}, \]
    that is, the set of Hodge structures in $D_\Lambda$ for which elements of $N$ are Hodge classes. Consider the group 
    \[ G_N := \{ g \in \mathrm{O}(\Lambda_\mathbb{R}) \mid \forall n \in N, g \cdot n = n \}. \]
    One has \begin{itemize}
        \item $G_N$ acts transitively on $D_N$ ;
        \item $G_N^0$ acts transitively on the connected components of $D_N$.
    \end{itemize}
\end{lem}

\begin{proof}
    Let us start by showing the first assertion. Let $T := N^\perp \subset \Lambda_\mathbb{Q}$. We distinguish two cases : either $q|_N$ is degenerate or it is not. 
    
    We start with the case where $q|_N$ is non-degenerate. In that case, one has $\Lambda_\mathbb{Q} = T \oplus N$ and thus $G_N = \mathrm{O}(T)$. Notice that since $T$ contains positive 2-planes, $q|_T$ is either of signature $(2,n)$ or $(3,n)$. In both cases, the group $\mathrm{O}(T)$ acts transitively on oriented positive 2-planes inside $T$, see for example \cite[Section 6.1]{Huybrechts_2016}.
    
    Let us now assume that $q|_N$ is degenerate. Let us also assume that $D_N$ is non-empty, for otherwise the assertion is automatically true. Let 
    \[ N_0 := \ker(q|_N) = \{ x \in N \mid \forall y \in N, q(x,y) = 0 \}. \]
    Let $N'$ be a supplement of $N_0$ in $N$. By definition, $N_0$ and $N'$ are orthogonal and $q|_{N'}$ is non-degenerate. Let $P \in D_N$, where we see $P$ as an oriented positive 2-plane in $\Lambda_\R$. The form $q|_{P^\perp}$ is non-degenerate and of signature $(1,p)$ for some $p \geqslant 0$. Notice that $N \subset P^\perp$. It is a classical fact that if $u_1, \dots, u_k$ is a basis of $N_0$, one can find vectors $v_1, \dots, v_k \in P^\perp$ such that for all $i \in \{1, \dots, k \}$, the planes $P_i := <u_i, v_i>$ are such that : \begin{itemize}
        \item $P_i$ is a hyperbolic plane and $\{u_i, v_i \}$ is a hyperbolic basis ;
        \item The sum $P_1 + \dots + P_k + N'$ is direct and orthogonal.
    \end{itemize}
    Here, the terms hyperbolic plane and hyperbolic basis mean that the restriction of $q$ to $P_i$ is non-degenerate and of signature $(1,1)$ and that $u_i$ and $v_i$ are isotropic vectors. Since $q|_{P^\perp}$ is of signature $(1,p)$, one must have $k = 1$ and $q|_{N'}$ negative definite. Fix a non-zero vector $n_0 \in N_0$, and let $h_P \in P^\perp$ be such that $\{n_0, h_P\}$ is a hyperbolic basis of the plane $H_P := <n_0, h_P>$. We thus have an orthogonal decomposition 
    \[ \Lambda_\R = T'_P \oplus H_P \oplus N', \]
    where $T'_P$ is a non-degenerate subspace containing $P$. Now, if $Q$ is another element of $D_N$, construct $T'_Q$, $h_Q$ and $H_Q$ as above. Consider the linear map $\varphi = \varphi_T \oplus \varphi_H \oplus \Id_{N'}$ with : \begin{itemize}
        \item $\varphi_T : T'_P \rightarrow T'_Q$ is an isometry sending $P$ to $Q$ that preserves their chosen orientations ;
        \item $\varphi_H : H_P \rightarrow H_Q$ is the isometry sending $h_P$ to a multiple of $h_Q$ and preserving $n_0$.
    \end{itemize}
    Clearly, $\varphi \in G_N$ and $\varphi(P) = Q$.

    Let us now show the second assertion. It suffices to show that the orbits of $G_N^0$ are open inside $D_N$. Since $G_N$ is an algebraic group, it has finitely many connected components. In particular, there exists $g_1, \dots, g_n \in G_N$ such that for $x \in D_N$ one has 
    \[ D_N = \bigcup_{i = 1, \dots, n} G_N^0 \cdot g_i \cdot x. \]
    From this, one deduces that for at least one $i$, and thus for all, the interior of $G_N^0 \cdot g_i \cdot x$ inside $D_N$ is non-empty. In particular, $G_N^0 \cdot x$ has non-empty interior and thus, by homogeneity, it is open.
\end{proof}

\begin{rmk}
    In the case where $N$ contains an element of positive square, Lemma \ref{lem domaine de Mumford Tate} above can be seen as a consequence of \cite[Lemma 9]{VanGeemenVoisin}.
\end{rmk}

The next lemma is an adaptation of \cite[Lemma 4.5]{GrebSchwald} to the case of non-necessarily compact surfaces.

\begin{lem}\label{lemme topologique}
    Let $S$ be a connected and second countable differentiable surface, $(U_i)_{i \in I}$ an open covering of $S$ and $B \subset X$ a subset containing at most countably many points. Then there exists $\psi : B \rightarrow I$ and a locally finite closed cover $(F_j)_{j \in J}$ of $S$ such that \begin{enumerate}
        \item\label{item 1 lem} $\forall x \in B, x \in U_{\psi(x)}$ ;
        \item\label{item 2 lem} For all $j \in J$, there exists $i \in I$ such that $F_j \subset U_i$ ;
        \item\label{item 3 lem} For all $j \in J$, $F_j \cap B$ is an open subset of $B$ and $\psi$ is constant on it. 
    \end{enumerate}
\end{lem}

\begin{proof}
    By Whitney's theorem, one has an embedding $S \rightarrow \R^4$. From now on, we will identify $S$ with its image under this embedding. We will assume that $0 \in S$. We consider the distance on $S$ induced by the standard distance in $\R^4$, and we will denote by $B(r)$ (resp. $S(r)$) the open ball (resp. sphere) of radius $r$ and centered at 0 in $S$. Since $B$ is at most countable, one can consider an increasing sequence $(r_n)_{n \in \mathbb{N}}$ such that $r_n \rightarrow +\infty$ and $S(r_n) \cap B = \emptyset$ for all $n$. When restricted to $B(r_n)$, the cover $(U_i)_{i \in I}$ has a Lebesgue number that we will denote by $\delta_n$. Recall that this means that if $ 0 < \delta < \delta_n$, then any ball of radius $\delta$ in $B(r_n)$ is contained in a $U_i$ for some $i \in I$. Consider $l_n > 0$ to be such that $[-\frac{l_n}{2}, \frac{l_n}{2}]^4 \subset B(\delta_n)$. Let $p_1, \dots, p_4$ be the coordinate projections on $\R^4$. Since $B$ is at most countable one can choose $\lambda_1, \dots, \lambda_{N_n} \in [-r_n, r_n]$, with $\lambda_1 < \dots < \lambda_{N_n}$, such that : \begin{itemize}
        \item $p_i^{-1}(\lambda_k) \cap B = \emptyset$ for all $i \in \{1, \dots, 4\}$ and $k \in \{1, \dots, N_n\}$ ; 
        \item $\lambda_{k+1} - \lambda_k < l_n$ for $k \in \{1, \dots, N_n -1 \}$.
    \end{itemize}
    
    Consider the set
    \[ B^0_n :=  B(r_n) \setminus \bigcup_{\substack{i \in \{1, \dots 4 \}, \\ k \in \{1, \dots N_n \}}}p_i^{-1}(\lambda_k).\]
    By construction, $B(r_n) \cap B = B^0_n \cap B$. Moreover, every connected component of $B^0_n$ is contained inside an open $U_i$ for some $i \in I$. For each such component, we choose such an $i \in I$. We construct a map $\psi_n : B(r_n) \cap B \rightarrow I$ by assigning to each element of $x \in B(r_n) \cap B$ the element in $I$ associated to the connected component of $B^0_n$ containing $x$. We now construct $\psi : B \rightarrow I$ as follows : $\psi|_{B(r_0) \cap B} = \psi_0$ and $\psi|_{(B(r_{n+1}) \setminus \overline{B}(r_n)) \cap B} = \psi_{r_{n+1}}|_{(B(r_{n+1}) \setminus \overline{B}(r_n)) \cap B}$ for $n \geqslant 0$. By our assumption on the sequence $(r_n)_{n \in \mathbb{N}}$ the map $\psi$ is well-defined. Let $(F_j)_{j \in J}$ be the set of closed subsets that are obtained as the closure of connected components of $B_{n+1}^0 \setminus \overline{B}(r_n)$ for some $n$. Then $(F_j)_{j \in J}$ is a locally finite closed cover of $S$ and together with $\psi$ they satisfy properties \ref{item 1 lem}, \ref{item 2 lem} and \ref{item 3 lem}.
\end{proof}

\begin{proof}[Proof of Theorem \ref{thm lift}]
    As mentioned in Section \ref{section espace de modules}, Markman has constructed in \cite{MarkmanSurvey} a family over $\mathcal{M}_\Lambda$ whose period map is $\mathcal{P}$. In particular, in order to show our assertion it suffices to construct a holomorphic map $\widetilde{f} : C \rightarrow \mathcal{M}^0_\Lambda$ such that $\widetilde{f}(x_0) = (X_0, \varphi_0)$ and $f= \mathcal{P} \circ \widetilde{f}$. We will refer to such a map as a lift of $f$.
    
    Let us start by noticing that since $f$ is holomorphic and $\mathcal{P}$ is étale, any continuous lift of $f$ will be holomorphic. Consider 
    \[ \Delta := \{\alpha \in \Lambda \mid \forall x \in C, \alpha \in \mathrm{MBM}(f(x))\}.\]
    We define 
    \[\Delta^+ := \{\alpha \in \Delta \mid \forall v \in \varphi_0(\Kah(X_0)), q(\alpha, v) > 0\}\]
    and 
    \[\Delta^- := \{\alpha \in \Delta \mid \forall v \in \varphi_0(\Kah(X_0)), q(\alpha, v) < 0\}.\]
    Let $N := \mathrm{span}_\mathbb{Q}(\Delta)$ and consider $D_N \subset D_\Lambda$ as in Lemma \ref{lem domaine de Mumford Tate}. Consider the following space 
    \[\widetilde{D}_N := \left(\mathcal{C}^+|_{D_N} \setminus \bigcup_{\alpha \in \Delta} \alpha^\perp \right)/ \sim\]
    where for two elements $a,b \in \mathcal{C}^+|_{D_N}$, $a \sim b$ if and only if they lie over the same point $x \in C$ and are in the same connected component of the decomposition 
    \[ \mathcal{C}^+_x \setminus \bigcup_{\alpha \in \Delta} \alpha^\perp.\]
    One has two projections $\mathcal{P}^{-1}(D_N) \xrightarrow{\pi_1} \widetilde{D}_N \xrightarrow{\pi_2} D_N$, where $\pi_2$ is natural projection map and $\pi_1$ is the map given by $(X, \varphi) \mapsto \overline{\varphi(\Kah(X))}$, which is well-defined by Theorem \ref{thm Torelli Global}.

    Let $D_N^0$ be the connected component of $D_N$ containing $f(C)$ and let $Ch_0$ be the element of $\pi^{-1}_2(f(x_0))$ containing $\varphi_0(\Kah(X_0))$. By Lemma \ref{lem domaine de Mumford Tate}, the group $G_N^0$ acts transitively on $D_N^0$. Clearly, $G_N^0 \subset \SO^0(\Lambda_\R)$. Thus, by Remark \ref{rmk action sur le cone}, the elements of $G_N^0$ preserve $\mathcal{C}^+$. Let $p \in D_N$ and $g_p \in G_N^0$ be such that $g_p \cdot f(x_0) = p$. Then, by construction, $g_p(Ch_0)$ is an element of $\pi_2^{-1}(p)$ such that for all $x \in g_p(Ch_0)$, one has \begin{itemize}
        \item $q(x, \delta^+) > 0$ for all $\delta^+ \in \Delta^+$ ; 
        \item $q(x, \delta^-) < 0$ for all $\delta^- \in \Delta^-$.
    \end{itemize}
    Clearly, if an element of $\pi_2^{-1}(p)$ satisfies the above condition it is unique. The map 
    \[ s : D_N^0 \rightarrow \widetilde{D}_N, \quad p \mapsto g_p(Ch_0) \]
    is a section of $\pi_2$ over $D_N^0$ such that $s(f(x_0)) = Ch_0$.
    
    We define
    \[B := \{c \in C \mid \mathrm{MBM}(f(c)) \neq \Delta \}.\]
    The set $B$ is the union of preimages of divisors of the form $\alpha^\perp$ for $\alpha \in \mathrm{MBM}(\mathcal{M}^0_\Lambda) \setminus \Delta$. In particular, $B$ is at most countable. By construction, $\pi_1$ is a bijection over $\pi_2^{-1}(C \setminus B)$.
    
    For $x \in C$, let $(X_x, \varphi_x) \in \mathcal{P}^{-1}(f(x))$ be such that $\pi_1(X_x, \varphi_x) = s(f(x))$. Such a pair exists by Theorem \ref{thm Torelli Global}. For $x_0$, choose $(X_0, \varphi_0)$. By considering the Kuranishi family of the pair $(X_x, \varphi_x)$, one obtains an open neighborhood $U_x$ of $f(x)$ in $D_N$ and a section $s'_x : U_x \rightarrow \mathcal{P}^{-1}(D_N)$ such that $\pi_1 \circ s'_x = s|_{U_x}$. Let $(V_x)_{x \in C}$ be the induced cover of $C$. Up to shrinking the open sets $V_x$ for $x \neq x_0$, we can assume that $V_{x_0}$ is the only element of the cover $(V_x)_{x \in C}$ containing $x_0$. For $x \in C$, consider the map $\widetilde{f}_x := s'_x \circ f$ defined on $V_x$. Notice that by construction, for $x_1,x_2 \in C$ and $y \in (U_{x_i} \cap U_{x_j}) \setminus B$, one has $\widetilde{f}_{x_i}(y) = \widetilde{f}_{x_j}(y)$. 
    
    Let $\psi$ and $(F_j)_{j \in J}$ be the data provided by Lemma \ref{lemme topologique} above with respect to the covering $(V_x)_{x \in C}$. We construct $\widetilde{f}$ as follows : 
    \[   
    \widetilde{f}(x) = 
     \begin{cases}
        \widetilde{f}_x(x)\quad\text{if }x \in C \setminus B\\
        \widetilde{f}_{\psi(x)}(x)\quad\text{if }x \in B
     \end{cases}.
    \]
    Since $(F_j)_{j \in J}$ is locally finite, the pasting lemma implies that $\widetilde{f}$ is continuous. Moreover, if $x_0 \notin B$, it is clear that $\widetilde{f}(x_0) = (X_0, \varphi_0)$. Assume now that  $x_0 \in B$. Since $U_{x_0}$ is the only element of the cover that contains $x_0$, one must have $\psi(x_0) = x_0$, and thus $\widetilde{f}(x_0) = (X_0, \varphi_0)$.
\end{proof} 

\begin{rmk}
    In Theorem \ref{thm lift} we do not make any assumption on the rank of $\Lambda$. If $\mathrm{rk}(\Lambda) \geqslant 6$, it is known that the elements of $\mathrm{MBM}(\mathcal{M}_\Lambda^0)$ have bounded square by \cite[Theorem 5.3]{AmerikVerbitskyConeConjecture} and thus that the collection of hyperplanes of the form $\alpha^\perp$ in Theorem \ref{thm AV decomp cone} is locally finite (\cite[Remark 3.16]{BakkerLehnJEMS}). In particular, in this case, one can show that $\mathcal{M}_\Lambda^0$ is homeomorphic to the space 
    \[\widetilde{D}_\Lambda := \left(\mathcal{C}^+|_{D_N} \setminus W \right)/ \sim\]
    where $W$ is the union of the walls of the form $\mathcal{C}^+_p \cap \alpha^\perp$ for $\alpha \in \mathrm{MBM}(p)$ and $\sim$ is the equivalence relation identifying two points that lie over the same point of $D_\Lambda$ and that are in the same chamber (\cite[Theorem 5.5]{BakkerLehnJEMS}). Here, $\widetilde{D}_\Lambda$ is endowed with its natural quotient topology. In the case where $\mathrm{rk}(\Lambda) \leqslant 5$, it is not clear to us whether $\mathcal{M}_\Lambda^0$ and $\widetilde{D}_\Lambda$ are homeomorphic. 
\end{rmk}

\section{Kobayashi pseudo-distance}

\subsection{Positive Kobayashi pseudo-distance}\label{section positive distance}

As in Section \ref{section lifting curves}, we fix a lattice $\Lambda$ such that $\mathcal{M}_\Lambda$ is non-empty. We introduce a variant of the Kobayashi pseudo-distance on $\mathcal{M}_\Lambda$. We refer the reader to \cite{kobayashi2013hyperbolic} for a general definition and discussion around the Kobayashi pseudo-distance. Let us denote by $\mathbb{P}(\Omega^1_D)^+$ the open subset of $\mathbb{P}(\Omega^1_D)$ (seen as the projectivized bundle of lines inside $T_D$) formed by lines on which $\omega_D$ restricts to a positive metric. Let $\Delta_r$ be the disk of radius r inside $\C$ and $\Delta$ be the disk of radius 1.  

\begin{df}
    A non-constant holomorphic map $f : \Delta \rightarrow D_\Lambda$ will be said to be \textit{positive} if its lift $f' : \Delta \rightarrow \mathbb{P}(\Omega^1_D)$ takes its values inside $\mathbb{P}(\Omega^1_D)^+$. A non-constant holomorphic map $f : \Delta \rightarrow \mathcal{M}_\Lambda$ will be said to be \textit{positive} if it is the moduli map of a $\Lambda$-marked family on $\Delta$ whose period map is positive.
\end{df}

\begin{df}
    The \textit{positive Kobayashi pseudo-distance} on $\mathcal{M}_\Lambda$ (resp. $D_\Lambda$) is the largest pseudo-distance $d_p$ on $\mathcal{M}_\Lambda$ (resp. $D_\Lambda$) such that for any positive holomorphic disk $f : \Delta \rightarrow \mathcal{M}_\Lambda$ one has $f^*d_p \leqslant \delta$, where $\delta$ is the Poincaré metric on $\Delta$. 
\end{df}

\begin{rmk}
    In \cite{DemHyp}, Demailly introduces the notion of directed manifold, that is, a pair $(X, V)$ where $X$ is a complex manifold and $V$ is a holomorphic subbundle of $T_X$. Most notions of hyperbolicity can be adapted in this framework. For example, the Kobayashi pseudo-distance can be defined by imposing that the holomorphic disks considered have to be tangent to $V$. In that regard, our notion of positive Kobayashi pseudo-distance can be seen as a version of Demailly's, where instead of considering subbundles of $T_X$ we consider an open subspace of $\mathbb{P}(\Omega^1_X)$.
\end{rmk}

As explained in \cite{kobayashi2013hyperbolic}, this pseudo-distance can be computed as follows. Let $p,q \in D_\Lambda$. A \textit{chain of positive holomorphic disks} $\alpha$ from $p$ to $q$ is the data of points $p_0, \cdots, p_k$ such that $p_0 = p$ and $p_k = q$, points $a_1, b_1, \cdots, a_k, b_k \in \Delta$ and positive disks $f_1, \cdots, f_k$ from $\Delta$ to $D_\Lambda$ such that $f_i(a_i) = p_{i-1}$ and $f_i(b_i) = p_i$ for all $1 \leqslant i \leqslant k$. The \textit{length} of $\alpha$ is defined to be 
\[l(\alpha) := \sum_{i=1}^k \delta(a_i, b_i).\]
One has
\[d_p(p,q) = \inf_\alpha l(\alpha)\]
where the infimum if taken over all chains of positive holomorphic disks from $p$ to $q$. 

\subsection{Vanishing of the positive Kobayashi pseudo-distance}

\begin{thm}\label{thm distance koba espace de mod}
    Assume that $\mathrm{rk}(\Lambda) \geqslant 4$. Then, the positive Kobayashi pseudo-distance vanishes on every connected component of $\mathcal{M}_\Lambda$. Moreover, to compute this pseudo-distance, one can restrict to using families $\pi : \mathcal{X} \rightarrow \Delta$ such that there exists a smooth family of \HK manifolds $\widetilde{\pi} : \widetilde{\mathcal{X}} \rightarrow \mathbb{P}^1$ with $\pi = \widetilde{\pi}|_{\widetilde{\pi}^{-1}(\Delta)}$.
\end{thm} 

The strategy that we will follow to prove Theorem \ref{thm distance koba espace de mod} is the following. We start by studying some rational curves inside 2-dimensional period domains of $\mathrm{K3}$-type. From this, we will deduce that the statement analogous to Theorem \ref{thm distance koba espace de mod} holds for any period domains of $\mathrm{K3}$-type (Proposition \ref{thm distance koba domaine}). Finally, we use Theorem \ref{thm lift} in order to deduce Theorem \ref{thm distance koba espace de mod} from Proposition \ref{thm distance koba domaine}. Throughout this section, we will use the notations introduced in the previous sections.

\begin{lem}\label{lem disques}
    Let $\lambda \in \C \setminus \R_+$ and consider the map 
    \[f_\lambda : \mathbb{P}^1 \rightarrow \check{D}_2, \quad [x_0:x_1] \mapsto ([x_0:x_1], [\lambda x_0:x_1]).\]
    Then $f_\lambda(\mathbb{P}^1) \subset D_2$. 
\end{lem}

\begin{proof}
    Let $p := [x_0:x_1] \in \mathbb{P}^1$. The assertion $f_\alpha(p) \in B$ is equivalent to $[\lambda x_0:x_1] = [\overline{x_1}:\overline{x_0}]$ and thus $\lambda |x_0|^2 = |x_1|^2$ which cannot be satisfied since $\lambda \notin \R_+$.
\end{proof}

\begin{lem}\label{lem disques positifs}
    Let $\alpha \in \C \setminus \R_+$ be such that $\mathrm{Re}(\alpha) > 0$ and let $n$ be a positive integer. Let $\alpha_n := \frac{\alpha}{n}$. Then under the identification $\mathbb{P}^1 = \C \cup \{\infty\}$, $f_{\alpha_n}|_{\Delta_r}$ is positive for $r = C_\alpha \sqrt{n}$, where $C_\alpha$ is a positive constant depending on $\alpha$.
\end{lem}

\begin{proof}
    Let $f_\lambda$ be such as in Lemma \ref{lem disques}. When restricted to $\C$, $f_\lambda$ takes values inside the open chart $\C \times \C \subset \mathbb{P}^1 \times \mathbb{P}^1$ introduced in Proposition \ref{metric dim 2}. In the associated coordinate system one has $f(t) = (t, \lambda t)$. 
    By Proposition \ref{metric dim 2} one has 
    \begin{align*}
        ||f_\lambda'(t)||^2_{\omega_D} &= 2 \mathrm{Re}\left(\frac{\lambda}{(\lambda |t|^2 - 1)^2} \right) \\
        & = 4 \frac{\mathrm{Re}(\lambda) |\lambda|^2 |t|^4 - 2 |\lambda|^2 |t|^2 + \mathrm{Re}(\lambda)}{|\lambda |t|^2 - 1|^4}.
    \end{align*}
    Consider the real polynomial $P(s) := \mathrm{Re}(\lambda) |\lambda|^2 s^2 - 2 |\lambda|^2 s + \mathrm{Re}(\lambda)$. The associated discriminant is 
    \[\Delta = 4 |\lambda|^4  - 4 \mathrm{Re}(\lambda)^2 |\lambda|^2 = 4 |\lambda|^2 \mathrm{Im}(\lambda)^2\]
    The roots of $P$ are therefore 
    \[ r= \frac{|\lambda| \pm |\mathrm{Im}(\lambda)|}{|\lambda|\mathrm{Re}(\lambda)}.\]
    Now, if one replaces $\lambda$ by $\alpha_n$, one has
    \[r = \frac{|\alpha| \pm |\mathrm{Im}(\alpha)|}{|\alpha|\mathrm{Re}(\alpha) }n.\]
    Therefore, since $\mathrm{Re}(\alpha) > 0$, for any $t \in \C$ such that $|t|^2 < \frac{|a| - |\mathrm{Im}(\alpha)|}{|a|\mathrm{Re}(\alpha)}n$ one has $||f_\alpha'(t)||^2_\omega > 0$. Moreover, $C_\alpha  := \frac{|a| - |\mathrm{Im}(\alpha)|}{|a|\mathrm{Re}(\alpha)} > 0$.
\end{proof}

\begin{lem}\label{distance p,q}
    Let $p := ([0:1], [0:1])$ and $q := ([1:1], [0:1])$. Then $d_p(p,q) = 0$. Moreover, this can be computed using chains of length at most 2 of positive disks that extend to rational curves inside $D_2$.
\end{lem}

\begin{proof}
    Consider the matrix 
    \[
    A := \begin{pmatrix}
        1 & 1 \\
        0 & 1
    \end{pmatrix} \in \mathrm{SL}_{2}(\C). 
    \]
    One has $A \cdot p = q$. Let $\alpha, \beta \in \C \setminus \R_+$ be such that $\mathrm{Re}(\alpha)>0$ and $\mathrm{Re}(\beta)>0$. Consider $f_\alpha$ and $g_\beta := A \cdot f_\beta$. Explicitly, one has $g_\beta(s) = ([s+1 : 1], [\beta s : \beta s + 1])$. Let $s, t \in \C$ be such that $f_\alpha(t) = g_\beta(s)$. Notice that one has $t = s + 1$ and that $s$ must satisfy the equation 
    \begin{equation}\label{equation}
        \alpha \beta s^2 + (\alpha \beta + \alpha - \beta)s + \alpha = 0
    \end{equation}
    to which a solution is
    \[s = \frac{\beta - \alpha( 1 + \beta) - r}{2 \alpha \beta}\]
    where $r$ is such that $r^2 = (\alpha \beta + \alpha - \beta)^2 - 4 \alpha^2\beta$. 
    Now, replace $\alpha$ by $\alpha_n = \frac{\alpha}{n}$ for some positive integer $n$. One has 
    \[r^2 = \beta^2 - \frac{2\beta(\alpha\beta + \alpha)}{n} + \frac{(\alpha\beta + \alpha)^2 - 4 \alpha^2\beta}{n^2}.\]
    Choose a determination of the square root in a neighborhood of $\beta^2$. We get
    \[ r = \beta - \frac{\alpha(1+ \beta)}{n} + o(\frac{1}{n}),\]
    which shows that one can choose $s$ and $t$ to be such that $s$ tends to 0 and $t$ tends to 1 as $n$ tends to $\infty$. By Lemma \ref{lem disques positifs}, $f_{\alpha_n}(C_\alpha \sqrt{n} \cdot)|_\Delta$ and $g_\beta|_\Delta$ form a chain of positive disks whose length tends to 0 as $n$ tends to $\infty$.
\end{proof}

\begin{prop}\label{prop dist koba dim 2}
    Let $D_2$ be a 2-dimensional period domain of $\mathrm{K3}$-type. Then the positive Kobayashi pseudo-distance vanishes on $D_2$. Moreover, this pseudo-distance can be computed using chains of positive disks that extend to rational curves inside $D_2$.
\end{prop}

\begin{proof}
    The stabilizer of $([0:1], [0:1]) \in D_2$ in $\SL_{2}(\C)$ consists of diagonal matrices, and is therefore isomorphic to $\C^*$. The orbit of $([1:1], [0:1])$ under this subgroup is the set of points of the form $([a:b], [0:1])$ with $a$ and $b$ both non-zero. Since $([1:0], [0:1]) \notin D_2$, Lemma \ref{distance p,q} implies that the positive Kobayashi pseudo-distance vanishes along $\pr_2^{-1}([0:1])$, where $\pr_2 : D_2 \rightarrow \mathbb{P}^1$ is the projection onto the second factor. 

    Notice that by definition, isometries of $(D, \omega)$ are isometries for $d_p$. By Proposition \ref{description action D2} the action of $G$ on $\mathbb{P}^1 \times \mathbb{P}^1$ is diagonal. Therefore, the property expressed above is true along all fibers of $\pr_2$. By Corollary \ref{coro involution isométrie} fibers of the projection on the first factor have the same property. We conclude using the triangular inequality.
\end{proof}

\begin{prop}\label{thm distance koba domaine}
    Let $D$ be a period domain of $\mathrm{K3}$-type with $\dim(D) \geqslant 2$. Then the positive Kobayashi pseudo-distance vanishes on $D$. Moreover, this pseudo-distance can be computed using chains of positive disks that extend to rational curves inside $D$.
\end{prop} 

\begin{proof}
    Let $p,q \in D$. By \cite[Proposition 3.7]{HuybrechtsBourbaki}, there exists a chain of twistor lines connecting $p$ and $q$. This means that there exists positive 3-planes $W_1, \dots, W_n \subset V$, and points $p = x_1, \dots, x_{k+1}=q \in D$ such that for all $i \in \{0, \dots k\}$, $x_i, x_{i+1} \in T_{W_i}$, where $T_{W_i}$ is the twistor line constructed in Example \ref{exemple twistor}. All the positive 3-planes $T_{W_i}$ are contained in a 4-plane of signature $(3,1)$. In particular, all the twistor lines $T_{W_i}$ are contained in a 2-dimensional subdomain of $D$. By Proposition \ref{prop dist koba dim 2}, one has $d_p(x_i, x_{i+1}) = 0$ and thus, by the triangular inequality, $d_p(p,q) = 0$. 
\end{proof}

Theorem \ref{thm distance koba espace de mod} is deduced from Proposition \ref{thm distance koba domaine} using the following lemma.

\begin{lem}\label{lem points non sep}
    Let $\mathcal{M}^0_\Lambda$ be a connected component of $\mathcal{M}_\Lambda$ and $p,q \in \mathcal{M}^0_\Lambda$ such that $\mathcal{P}(p) = \mathcal{P}(q)$. Then $d_p(p,q) = 0$.
\end{lem}

\begin{proof}
    Let $r \in D_\Lambda$ be the common period point. Take a positive disk $f : \Delta \rightarrow D_\Lambda$ such that $f(0) = r$ and whose image is not contained inside a divisor of the form $\alpha^\perp$ for $\alpha \in \Lambda$. Using Theorem \ref{thm lift}, consider two lifts $f_1,f_2 : \Delta \rightarrow \mathcal{M}^0_\Lambda$ such that $f_1(0) = p$ and $f_2(0) =q$. By our assumption on $f$, $f_1$ and $f_2$ take the same values on the complementary of a countable set points in $\Delta$, which is dense in $\Delta$. By choosing a sequence of points in this complementary converging to $0$, $f_1$ and $f_2$ provide a chain of positive disks whose length tends to 0. 
\end{proof}

\begin{proof}[Proof of Theorem \ref{thm distance koba espace de mod}]
    Let $\mathcal{M}^0_\Lambda$ be a connected component of $\mathcal{M}_\Lambda$. Let $p,q \in \mathcal{M}^0_\Lambda$, $p':= \mathcal{P}(p)$, $q':= \mathcal{P}(q)$ and $\varepsilon > 0$. By Proposition \ref{thm distance koba domaine}, we have a chain of positive disks $f_1, \dots,f_n : \Delta \rightarrow D_\Lambda$, with $a_1,b_1, \dots, a_n, b_n$ as above, from $p'$ to $q'$ whose length is smaller than $\varepsilon/2$. By Theorem \ref{thm lift}, these disks lift to positive disks $\widetilde{f}_1, \dots, \widetilde{f}_n : \Delta \rightarrow \mathcal{M}^0_\Lambda$ such that $\widetilde{f}_1(a_1) = p$ and for $i \in \{2, \dots, n\}$, $\widetilde{f}_i(a_i) = \widetilde{f}_{i-1}(b_{i-1})$. Since $\widetilde{f}_n(b_n)$ and $q$ have the same period point, Lemma \ref{lem points non sep} implies that there exists a chain of two positive disks from $\widetilde{f}_n(b_n)$ to $q$ of length $\varepsilon/2$. This concludes the proof.
\end{proof}

\section{Restriction on entire curves inside period domains}\label{section 6}

In this section we formulate restrictions on entire curves $f : \C \rightarrow D$ that are tangent to directions where $\omega_D$ is positive. We will actually consider the more general situation of entire curves in the domain $\Omega$ introduced in Definition \ref{def domaine Omega}. 

\subsection{Nevanlinna theory}\label{sec Nevanlinna}

One of the main tools for our study will be Nevanlinna theory. In this section we gather a few facts about this theory. We refer the reader to \cite{NoguchiWinkelmann} or \cite{DemaillyNevanlinna} for more details. We will only treat the case of entire curves. More generally, one can consider holomorphic maps from parabolic Riemann surfaces. We refer the reader to \cite{PaunSibony} or \cite{BrotbekBrunebarbe} for more details on this aspect.

\begin{df}
    Let $\alpha$ be a (1,1)-current of order $0$ on $\C$. The \textit{characteristic function} of $\alpha$ is  
    \[T_\alpha : [1, +\infty[ \rightarrow \R, \quad r \mapsto \int_1^r \frac{dt}{t} \int_{\Delta_t} \alpha.\]
    
\end{df}

Let us recall the Jensen formula (see for example \cite[Lemma 2.1.33]{NoguchiWinkelmann}).

\begin{prop}[Jensen formula]\label{Jensen}
    Let $\varphi : \C \rightarrow [-\infty, + \infty [$ be a function that, locally near every point of $\C$, can be written as a difference of two subharmonic functions. Then for any $r >1$, one has 
    \[ \int_1^r \frac{dt}{t} \int_{\Delta_t} i\ddbar \varphi = \int_0^{2\pi} \varphi(re^{i\theta})d\theta - \int_0^{2\pi} \varphi(e^{i\theta})d\theta.\]
\end{prop}

Let $X$ be a complex manifold and $\omega$ be a (1,1)-form on $X$. If $f : \C \rightarrow X$ is a holomorphic map, we will denote by $T_{f,\omega}$ the function $T_{f^*\omega}$. One has the following standard fact (see for example \cite[2.11 Cas "local"]{DemaillyNevanlinna}, or \cite[lemma 3.3]{CadorelDeng} for a Kähler version).

\begin{prop}\label{critère algébrique}
    Let $(X, \omega)$ be a compact Kähler manifold and let $f : \C \rightarrow X$ be a holomorphic map. Then $f$ extends to a morphism $\mathbb{P}^1 \rightarrow X$ if and only if $T_{f, \omega}(r) = O(\log(r))$.
\end{prop}

\subsection{Comparison of characteristic functions}

Let $h$ be a hermitian form on $\C^{3+p}$ of signature $(3,p)$. Consider the domain 
\[ \Omega := \{l \in \mathbb{P}^{p+2} \mid \text{$h$ is positive definite on $l$} \}\]
introduced in Definition \ref{def domaine Omega}. This domain carries a pseudo-Kähler metric of signature $(p,2)$ defined as the curvature form of the metric induced by $h$ on $\mathcal{O}_{\mathbb{P}^{p+2}}(-1)|_\Omega$. We will denote this metric by $\omega$. We want to exhibit restrictions on entire curves $f : \C \rightarrow \Omega$ that are tangent to directions on which $\omega$ is positive. 

\begin{rmk}
    More generally, one could consider the same problem for arbitrary numbers $p,q$, that is, consider $\C^{p+q}$ endowed with a hermitian form of signature $(p,q)$ and carry the same constructions. Conducting the same computations as in Proposition \ref{metric de GS}, one finds that the metric $\omega$ on the associated domain $\Omega$ is of signature $(q, p-1)$. We notice that for $p = 1$, $\Omega$ is the $q$-dimensional ball and $\omega$ is positive definite (in this case, it is a multiple of the Bergman metric on the ball). In particular, in this case, we know that there are no entire curves in $\Omega$. It is therefore natural to expect that, for arbitrary $(p,q)$, entire curves that are tangent to directions on which $\omega$ is positive are under strong constraints. Notice moreover that, since $\mathcal{O}_{\mathbb{P}^{p+q}}(-1)$ is anti-ample, if $C$ is a compact curve then there exists no holomorphic map $f : C \rightarrow \Omega$ such that $f^*\omega \geqslant 0$. 
\end{rmk}

Let $h_{FS}$ be a positive definite hermitian metric on $\C^{3+p}$. The curvature of the metric associated to $h_{FS}$ on $\mathcal{O}_{\mathbb{P}^{p+2}}(1)|_\Omega$ is a Fubini-Study metric restricted to the open $\Omega$. We will denote this metric by $\omega_{FS}$. Let $f : \C \rightarrow \Omega$ be a holomorphic map. We want to compare the Nevanlinna characteristic functions of $f$ with respect to $\omega$ and $\omega_{FS}$. To this end, we introduce the following function. 

\begin{df}
    We define the \textit{proximity to the boundary} by 
    \[p_{f}(r) := \int_0^{2\pi} \log(\varphi \circ f(re^{i\theta})) d\theta,\]
    where $\varphi$ is defined as the function on $\Omega$ satisfying $h_{FS} = \varphi h$ on $\mathcal{O}_{\mathbb{P}^{p+2}}(1)|_\Omega$.
\end{df}

\begin{rmk}
    By compactness of $\mathbb{P}^{p+2}$, this function only depends on the choice of the positive definite metric $h_{FS}$ on $\C^{3+p}$ up to a bounded function. In particular, for our application, the choice of a particular positive definite metric has no impact. 
\end{rmk}

The function $\varphi$ measures how close a point $p \in \Omega$ is to the boundary of $\Omega$ inside $\mathbb{P}^{2+p}$, hence the name of the function $p_f$.

\begin{prop}\label{comparaison Nevanlinna}
    One has 
    \[T_{f, \omega_{FS}}(r) + T_{f, \omega}(r) = p_f(r) + O(1).\]
\end{prop}

\begin{proof}
    By definition of $\varphi$, one has $\omega_{FS} = i\ddbar\log(\varphi) -\omega$. Since $h$ and $h_{FS}$ are both smooth and positive on $\Omega$, $\varphi$ is a smooth positive function on $\Omega$. In particular, $\log(\varphi \circ f)$ is smooth and thus locally the difference of two subharmonic functions. Proposition \ref{Jensen} gives
    \begin{align*}
        T_{f, \omega}(r) &= -T_{f, \omega_{FS}}(r) + \int_1^r \frac{\dd t}{t} \int_{\Delta_t} i\ddbar \varphi \\
        & = -T_{f, \omega_{FS}}(r) + p_f(r) + O(1).
    \end{align*}
\end{proof}

\begin{cor}\label{coro Nevanlinna 0jet}
    Assume that $f^*\omega \geqslant 0$. Then, if $p_f(r) = O(\log(r))$, $f$ extends to a morphism $\widetilde{f} : \mathbb{P}^1 \rightarrow \check{D}$. In particular, if $f$ is non-constant then $\varphi \circ f(z)$ is unbounded. 
\end{cor}

\begin{proof}
    Under our assumption one has $T_{f, \omega} \geqslant 0$ and the first assertion thus follows from Proposition \ref{critère algébrique}. Assume that $\varphi \circ f$ is bounded. Then $\varphi \circ f \circ \exp$ is also bounded and thus both $f$ and $f \circ \exp$ extend to algebraic morphisms from $\mathbb{P}^1$ to $\mathbb{P}^{p+2}$. This implies that $f$ is constant. 
\end{proof}

Assume that $h$ is the standard form of signature $(3,p)$ on $\C^{3+p}$. For $x \in \C^{3+p}$, one has $h(x) = \sum_{i = 1}^3 |x_i|^2 - \sum_{j = 4}^{p+3} |x_j|^2$. For $\varepsilon >0$, let $h_\varepsilon$ be the form of signature $(3,p)$ on $\C^{3+p}$ such that $h_\varepsilon(x) = (1+\varepsilon)\sum_{i = 1}^3 |x_i|^2 - \sum_{j = 4}^{p+3} |x_j|^2$. Remark that every line $l \in \Omega$ is positive for $h_\varepsilon$. Let $\omega_\varepsilon$ be the pseudo-hermitian metric on $\Omega$ associated to $h_\varepsilon$. 

Let $\mathbb{P}(\Omega^1_\Omega)^+ \subset \mathbb{P}(\Omega^1_\Omega)$ be the subset formed by lines inside $T_{\Omega}$ on which $\omega$ restricts to a positive definite metric, and let $\Omega_{1,\epsilon}$ be the set of lines on which $\omega_\varepsilon$ restricts to a positive definite metric. 

\begin{lem}
    For all $\varepsilon, \eta >0$ such that $\eta \geqslant \varepsilon$ one has $\Omega_{1,\epsilon} \subset \Omega_{1,\eta} \subset \mathbb{P}(\Omega^1_\Omega)^+$. Moreover, one has $\mathbb{P}(\Omega^1_\Omega)^+ = \bigcup_{\varepsilon >0} \Omega_{1,\varepsilon}$.
\end{lem}

\begin{proof}
    Follows directly from the description of $\omega$ given in Proposition \ref{metric de GS}. 
\end{proof}

Recall that if $f : C \rightarrow \Omega$ is a non-constant holomorphic map with $\dim(C) = 1$, then one obtains a lift $f' : C \rightarrow \mathbb{P}(\Omega^1_\Omega)$.

\begin{prop}\label{prop nevanlinna 1jet}
    Let $f : \C \rightarrow D$ be a non-constant holomorphic map. For all $\varepsilon > 0$ there exists $x \in \C$ such that $f'(x) \notin \Omega_{1,\varepsilon}$.
\end{prop}

\begin{proof}
    Assume that $h_{FS}$ is the standard positive definite hermitian form on $\C^{3+p}$. Consider $\varphi_\varepsilon := \frac{h_{FS}}{h_\varepsilon}$. Let $l \in \Omega$ and $x \in l$. One has 
    \[ \varphi_\varepsilon(l) = \frac{\sum^{3+p}_{k = 1}|x_k|^2}{(1+\varepsilon)\sum_{i = 1}^3 |x_i|^2 - \sum_{j = 4}^{p+3} |x_j|^2} \]
    Since $l \in \Omega$, one has $\sum_{i = 1}^3 |x_i|^2 > \sum_{j = 4}^{p+3} |x_j|^2$. In particular one has 
    \[\varphi_\varepsilon(l) \leqslant \frac{\sum_{i = 1}^3 |x_i|^2}{\varepsilon \sum_{i = 1}^3 |x_i|^2} = \frac{1}{\varepsilon}.\]
    If $f(\C) \subset \Omega_{1,\varepsilon}$ then $f^*\omega_\varepsilon \geqslant 0$. The assertion therefore results of Corollary \ref{coro Nevanlinna 0jet} applied to the metric $\omega_\varepsilon$. 
\end{proof}

\subsection{Second Main Theorem with error term}

Let $f : \C \rightarrow \Omega$ be an entire curve inside a period domain of $\mathrm{K3}$-type such that $f'(\C) \subset \mathbb{P}(\Omega^1_\Omega)^+$. Using techniques of \cite{BrotbekBrunebarbe}, we formulate an analogue of the Second Main Theorem in Nevanlinna theory for $f$ with an extra term carrying information on the increase or decrease of the curvature of $\omega$ on the image of $f$. 

Before doing so, we recall the notion of singular metrics for line bundles. We refer to \cite{DemaillyAnalyticMethods} for more details. Let $X$ be a complex manifold and $\mathcal{L}$ a line bundle on $X$.

\begin{df}
    A \textit{singular metric} $h$ on $\mathcal{L}$ is a metric such that if $\xi$ is a local trivialization of $\mathcal{L}$, i.e. a local never vanishing holomorphic section, defined on an open $U$ then $||\xi||^2_h = e^{-\varphi}$ with $\varphi \in L^1_{\mathrm{loc}}(U)$. The function $\varphi$ will be called a \textit{local weight} of $h$.
\end{df}

The assumption on the local integrability of the weights allows one to define the curvature of $h$ in the sense of currents by the usual formula ; if $\xi$ is a local trivialization then locally $i\Theta(\mathcal{L}, h) = -i \ddbar \log(||\xi||^2_h)$. 

Let $D$ be an effective divisor on a complex manifold $X$, and let $\mathcal{O}_X(D)$ be the associated line bundle. Recall that the sheaf of sections of $\mathcal{O}_X(D)$ is isomorphic to the subsheaf of meromorphic functions $u$ on $X$ such that $\mathrm{div}(u) + D \geqslant 0$. One defines the singular metric $h_D$ associated to $D$ on $\mathcal{O}_X(D)$ to be such that $||u||_{h_D} = |u|$. By the Lelong-Poincaré formula, one has 
\[i \Theta(\mathcal{O}_X(D), h_D) = 2\pi[D]\]
where $[D]$ is the current of integration associated to $D$. 

We now come back to our problem. We start by noticing the following. 

\begin{lem}\label{lem HSC}
    There exists $\gamma >0$ such that for all $u \in \mathbb{P}(\Omega^1_\Omega)^+$, $\HSC_\omega(u) = -\gamma$.
\end{lem}

\begin{proof}
    The group $G = \mathrm{SU}(3,p)$ acts transitively by holomorphic isometries on $\Omega$. This implies that $G$ also acts on $\mathbb{P}(\Omega^1_\Omega)^+$. Using the identification $T_{\Omega, l} \simeq \Hom(l, l^\perp)$, one sees that this action is transitive. In particular, $\HSC_\omega$ is the same for every element of $\mathbb{P}(\Omega^1_\Omega)^+$.

    The action of $G$ on $\Omega$ gives a diffeomorphisms $\Omega \simeq \mathrm{SU}(3,p)/ \mathrm{SU}(2,p)$. One therefore obtains that the $p$-dimensional ball $\mathbb{B}^p \simeq \mathrm{SU}(1,p)/\mathrm{SU}(p)$ embeds in $\Omega$ as a totally geodesic submanifold. Moreover, the restriction of $\omega$ to this ball coincides (up to multiplication by a positive constant) with the Bergman metric on the ball. In particular, the value of the holomorphic sectional curvature of $\omega$ is a positive multiple of that of the Bergman metric on the $p$-dimensional ball, which is known to be negative.
\end{proof}

\begin{rmk}
    The constant $\gamma$ above depends on $p$.
\end{rmk}

Let $f : \C \rightarrow \Omega$ be a non-constant holomorphic map such that $f'(\C) \subset \mathbb{P}(\Omega^1_\Omega)^+$. Define $\mathcal{L} := f'^*\mathcal{O}_{\mathbb{P}(\Omega^1_\Omega)}(-1)$. By construction, $\mathcal{L}$ is a sub-line bundle of $f^*T_\Omega$ on which $\omega$ is positive definite. Let $\beta \in \mathcal{A}^{1,0}(\Hom(\mathcal{L}, \mathcal{L}^\perp))$ be the second fundamental form associated to $\mathcal{L}$ with respect to $(f^*T_\Omega,f^*\omega)$, defined in Section \ref{section semi-herm met}. Since $T_\C$ is trivial, one has $\beta = \psi \otimes dz$ with $\psi \in \Hom(\mathcal{L}, \mathcal{L}^\perp)$. 

\begin{lem}
    For any $x \in \C$, $\psi_x$ depends only on the lift $f^{[2]}(x) \in P_2\Omega$, where $P_2\Omega$ is the second stage of the Demailly-Semple jet bundle tower over $\Omega$.
\end{lem}

We refer to \cite{DemHyp} for the definition and properties of the Demailly-Semple jet bundles. 

\begin{proof}
    Let $x \in \C$. First, assume that $f'(x) \neq 0$. There is a neighborhood $U$ of $x$ in $\C$ such that $df$ provides an identification between $T_\C|_U$ and $\mathcal{L}|_U$. Let $X := \frac{\partial}{\partial z}$ denote the standard vector field on $\C$, that we see as a section of $\mathcal{L}$ over $U$. Choose coordinates on $\Omega$ around $f(x)$. Using these coordinates, we write $f = (f_1, \dots, f_{p+2})$ and consider $f' := (f'_1, \dots, f'_{p+2})$, $f'' := (f''_1, \dots, f''_{p+2})$. For the trivialization of $T_\Omega$ induced by the choice of coordinates, the Chern connection $\nabla$ writes as $\nabla = d + A$ with $A$ a matrix of 1-forms. Let $\widetilde{A}$ be the matrix of functions on $U$ such that $f^*A = \widetilde{A}dt$. In the induced trivialization on $f^*T_\Omega$, one has 
    \[\beta(X) \cdot X = \pr_{\mathcal{L}^\perp}(f'' + \widetilde{A}f'). \]
    In particular, the value of $\beta$ at $x$ depends only on the 2-jet defined by $(f,x)$. Clearly, this value is invariant under reparametrization of the germ $(f,x)$. In particular, by \cite[Theorem 6.8]{DemHyp}, the value of $\beta$ at $x$ depends only on $f^{[2]}(x) \in P_2\Omega$.

    Assume now that $f'(x) = 0$. As above, choose coordinates on $\Omega$ centered at $f(x)$ and write $f = (f_1, \dots, f_{p+2})$. Then there exists $m \in \mathbb{N}$ such that $f' = (z-x)^mg$ with $g(x) \neq 0$. By construction, $(f, [f']) = (f, [g]) \in P_1\Omega$ and $g$ provides a non-vanishing section $X$ of $\mathcal{L}$ in a neighborhood of $x$. As above, one gets that 
    \[\beta(X) \cdot X = \pr_{\mathcal{L}^\perp}(g' + \widetilde{A}g), \]
    which, as above, concludes. 
\end{proof}

\begin{thm}\label{thm SMT}
    Let $f : \C \rightarrow \Omega$ be a non-constant holomorphic map such that $f'(\C) \subset \mathbb{P}(\Omega^1_\Omega)^+$. Let $\psi$ be as above and let $\sigma_f := ||\psi||_{f^*\omega}^2 \frac{i}{2} dz \wedge d\overline{z}$. One has 
    \[ \gamma T_{f, \omega}(r) + T_{\sigma_f}(r) \leqslant_{\mathrm{exc}} O(\log(r) + \log(T_{f, \omega}(r))),\]
    where the subscript $\mathrm{exc}$ stands for the fact that the inequality holds outside a subset of finite Lebesgue measure in $[1, + \infty [$.
\end{thm}

\begin{proof}
    Let $\Ram(f)$ be the ramification divisor of $f$, i.e. the zero divisor of the natural section of $\Hom(T_\C, \mathcal{L})$ given by $df$. By construction, if we see $T_\C \otimes \mathcal{O}(\Ram(f))$ as a subsheaf of the sheaf of meromorphic vector fields on $\C$, $df$ provides an isomorphism between $\mathcal{L}$ and $T_\C \otimes \mathcal{O}(\Ram(f))$. Let $X$ be a nowhere vanishing section of $\mathcal{L}$, whose existence is guaranteed by the fact that every line bundle on $\C$ is trivial, and let $g \frac{\partial}{\partial z} \in T_\C \otimes \mathcal{O}(\Ram(f))$ be the corresponding meromorphic vector field. Let $h_\omega$ be the hermitian metric induced by $\omega$ on $\mathcal{L}$. By Proposition \ref{courbure induite}, one has $-i\Theta(\mathcal{L}, h_\omega) = \lambda \frac{i}{2}dz \wedge d\overline{z}$ with \begin{align*}
        \lambda &= -\frac{h(\Theta(f^*T_\Omega, \omega)(\frac{\partial}{\partial z}, \frac{\partial}{\partial \overline{z}}) \cdot X,X)}{||X||^2_{f^*\omega}} + \frac{||\beta(X) \cdot X||^2_{f^*\omega}}{||X||^2_{f^*\omega}} \\
        &= - \HSC_\omega(X) \frac{||X||^2_{f^*\omega}}{|g|^2} + ||\psi||_{f^*\omega}^2 \\
        &= \gamma ||f'||^2_\omega + ||\psi||_{f^*\omega}^2
    \end{align*}
    In particular, one has 
    \[ -i\Theta(\mathcal{L}, h_\omega) = \gamma f^*\omega + \sigma_f.\]
    Under the identification $\mathcal{L} = T_\C \otimes \mathcal{O}(\Ram(f))$ one has $h_\omega = f^*\omega \otimes h_{\Ram(f)}$. In particular, in the sense of currents, one has 
    \[ -\mathrm{Ric}[f^*\omega] := -i\Theta(T_\C, f^*\omega) =  \gamma f^*\omega + \sigma_f + [\Ram(f)] \geqslant \gamma f^*\omega + \sigma_f.\]
    Applying \cite[Lemma 2.10]{BrotbekBrunebarbe}, one gets 
    \[ T_{-\mathrm{Ric}[f^*\omega]} \leqslant_{\mathrm{exc}} O(\log(r) + \log(T_{f, \omega}(r))),\]
    which, by the preceding inequality, concludes. 
\end{proof} 

\begin{cor}
    In the same setting as above, assume that there exists $\varepsilon >0$ such that $f^*\omega \geqslant \varepsilon f^*\omega_{FS}$. Then $T_{\sigma_f}(r)$ cannot be bounded from below.
\end{cor}

\begin{proof}
    By Theorem \ref{thm SMT}, one has
    \[ \gamma T_{f, \omega}(r) - O(\log(r) + \log(T_{f, \omega}(r))) \leqslant_{\mathrm{exc}} - T_{\sigma_f}(r).\]
    Since $T_{f, \omega}$ is non-negative, if $T_{\sigma_f}(r)$ was bounded from below then one would have $T_{f, \omega}(r) = O(\log(r))$. This would imply that $T_{f, \omega_{FS}}(r) = O(\log(r))$ and thus, arguing as in Corollary \ref{coro Nevanlinna 0jet}, $f$ would be constant. 
\end{proof}

\bibliographystyle{amsalpha}
\bibliography{biblio.bib}

\end{document}